\documentclass[11pt]{amsart}

\usepackage{amssymb, amsfonts, amsmath, amsthm,mathtools}
\usepackage{enumitem} 

\RequirePackage{hyperref}
\RequirePackage[small,bf]{caption}

\renewcommand{\emph}{\textbf}

\newcommand{\bmatr}[1]{\begin{bmatrix}#1\end{bmatrix}}

\newcommand{\bC}{\mathbb{C}}
\let\C\bC
\newcommand{\N}{\mathbb{N}} 
\newcommand{\R}{\mathbb{R}}
\newcommand{\bR}{\mathbb{R}}
\newcommand{\bS}{\mathbb{S}}

\newcommand{\range}{\mathrm{Range\,}}
\newcommand{\im}{\mathrm{Range\,}}

\newcommand{\Id}{I}

\renewcommand{\phi}{\varphi}
\renewcommand{\Re}{\mathrm{Re}\,} 
\renewcommand{\Im}{\mathrm{Im}\,}
\newcommand{\id}{\mathrm{id}}
\newcommand{\Herm}{\mathrm{Herm}}

\DeclareMathOperator{\rank}{rank}
\DeclareMathOperator{\tr}{tr}
\DeclareMathOperator{\Tr}{Tr}
\DeclareMathOperator{\diag}{diag} 
\DeclareMathOperator{\Capa}{Cap} 
\DeclareMathOperator{\DS}{DS}

\newcommand{\RR}{\mathcal{R}} 
\newcommand{\VV}{\mathcal{V}} 
\newcommand{\LL}{\mathcal{L}} 
\newcommand{\DD}{\mathcal D}
\newcommand{\PP}{\mathcal P} 

\def\sv{\boldsymbol s}

\theoremstyle{plain}
\newtheorem{theo}{Theorem}[section]
\newtheorem{lemma}[theo]{Lemma}
\newtheorem{propo}[theo]{Proposition}
\newtheorem{proposition}[theo]{Proposition}
\newtheorem{coro}[theo]{Corollary}

\theoremstyle{definition}
\newtheorem{defi}[theo]{Definition}

\theoremstyle{remark}
\newtheorem{remark}[theo]{Remark}
\newtheorem{exa}[theo]{Example}

\usepackage{tensor}
\usepackage[]{natbib}
\usepackage{tcolorbox}
\usepackage{mdframed}
\usepackage{tikz}
\usetikzlibrary{patterns,decorations.markings,calc,arrows.meta}

\definecolor{moved}{RGB}{0,100,180}
\definecolor{note}{RGB}{80,120,60}

\renewcommand{\phi}{\varphi}

\usepackage{titlesec}

\titleformat{\section}
  {\normalfont\Large\bfseries}{\S\thesection}{1em}{}
\titleformat{\subsection}
  {\normalfont\large\bfseries}{\S\thesubsection}{1em}{}
\titleformat{\subsubsection}
  {\normalfont\normalsize\bfseries}{\S\thesubsubsection}{1em}{}

\newtheorem*{mainthm}{Theorem}

\title{
An algebraic characterization of non-singular matrix semicircles
}
\bigskip
\author{
Vladislav Kargin
}
\thanks{email:
vkargin@binghamton.edu; current address: 4400 Vestal Pkwy East, Department of Mathematics, Binghamton University, Binghamton, 13902-6000, USA}

\begin{document}

\begin{abstract}
Let $A_1, \ldots, A_r$ be Hermitian $n \times n$ matrices and 
$S = \sum A_i \otimes s_i$ the associated matrix semicircle, 
where $s_1, \ldots, s_r$ are free semicircular variables. We 
prove that the following are equivalent: (i) the matrix pencil 
$A = \sum A_i x_i$ is LR-semisimple (decomposes, up to 
left--right equivalence, as a direct sum of unsplittable 
pencils); (ii) $S$ is non-singular at $t = 0$ (the 
matrix-valued Cauchy transform has a continuous boundary limit 
near the origin); (iii) the covariance map 
$\eta\colon X \mapsto \sum A_i X A_i$ is symmetrically 
DS-scalable (there exists $C \succ 0$ with 
$\eta(C) = C^{-1}$). When these hold, the spectral density 
satisfies $f(0) = \frac{1}{\pi}\,\mathrm{tr}(C)$, where $C$ 
is the unique trace minimizer of the solution set 
$\{W \succ 0 : \eta(W)\,W = I\}$.

The proof combines algebraic and analytic ingredients. On the 
algebraic side, we establish the equivalence (i) 
$\Leftrightarrow$ (iii) using Gurvits' capacity theory for 
indecomposable maps and a geodesic reflection theorem in the 
Riemannian manifold of positive definite matrices, which 
upgrades DS-scalability to symmetric DS-scalability for 
self-adjoint completely positive maps. On the analytic side, 
we prove (iii) $\Rightarrow$ (ii) via a Lyapunov--Schmidt 
reduction of Speicher's equation at a trace-minimizing 
solution, showing that the Jacobian of the bifurcation 
equations is positive definite. This removes a stability 
hypothesis that was required in earlier approaches.
\end{abstract}

\maketitle

\tableofcontents  

\section{Introduction} 
\label{section_introduction}

Let $A_1, \ldots, A_r \in M_n(\mathbb{C})$ be Hermitian matrices 
and let $s_1, \ldots, s_r$ be free standard semicircular variables 
in a non-commutative probability space $(\mathcal{A}, \varphi)$. 
The random variable $S = \sum_{i=1}^r A_i \otimes s_i$ is called 
a \emph{matrix semicircle}. Its spectral distribution encodes 
the interplay between the algebraic structure of the matrix 
pencil $A = \sum A_i x_i$ and the free probabilistic behavior 
of the semicircular variables. A natural question is: when is 
the spectral distribution of $S$ non-singular?

In \citet{msy2023}, it was shown that the distribution of $S$ 
has no atom at~$0$ if and only if the matrix pencil $A$ has 
full inner rank (equal to $n$). In this paper, we address the 
finer question of when the distribution has a bounded 
continuous density at $0$. The answer turns out to be 
governed by an algebraic property of the pencil that we call 
\emph{LR-semisimplicity}: $A$ is LR-semisimple if it 
decomposes, up to left--right equivalence, as a direct sum 
of unsplittable pencils.

To the Hermitian matrix pencil $A$ one associates the completely positive 
(CP) map $\eta_A\colon M_n(\mathbb{C}) \to M_n(\mathbb{C})$ 
defined by $\eta_A(X) = \sum_{i=1}^r A_i X A_i$, which is the 
covariance map of the semicircle $S$. We say that $\eta_A$ 
is \emph{symmetrically DS-scalable} if there exists a positive 
definite matrix $C \succ 0$ such that the rescaled map 
$X \mapsto C^{1/2}\,\eta_A(C^{1/2} X C^{1/2})\,C^{1/2}$ is 
doubly stochastic (i.e., unital and trace-preserving), or 
equivalently, $\eta_A(C) = C^{-1}$.

The spectral distribution of $S$ is encoded by its 
matrix-valued Cauchy transform 
$G_S(z) = \mathbb{E}[(zI \otimes 1_{\mathcal{A}} - S)^{-1}]$, 
and the spectral density is recovered from its boundary 
behavior via Stieltjes inversion. We say that $S$ is 
\emph{non-singular} if the matrix-valued Cauchy transform 
has a continuous non-tangential boundary limit on the real 
axis; in particular this implies that the spectral density 
$f(x) = \frac{1}{\pi}\,\Im[-\tr\, G_S(x + i0^+)]$ is 
bounded and continuous. The formal definition is given in 
Definition~\ref{defi_nonsingularity}.

\begin{samepage}
\bigskip
\noindent\textbf{Main Result.}
\begin{mainthm}[A]
Let $A = \sum_{i=1}^r A_i x_i$, $A_i^* = A_i$, be a Hermitian
matrix pencil, $S = \sum A_i \otimes s_i$ the associated matrix
semicircle, and $\eta\colon X \mapsto \sum A_i X A_i$ its
covariance map. The following are equivalent:
\begin{enumerate}[label=\textup{(\roman*)}]
\item $A$ is LR-semisimple.
\item $S$ is non-singular at $t = 0$.
\item $\eta$ is symmetrically DS-scalable.
\end{enumerate}
Moreover, when these hold, the spectral density satisfies
\[
  f(0) = \frac{1}{\pi}\,\tr(C)
  \;\geq\; \frac{1}{\pi}\,\Capa(\eta)^{-1/(2n)}
  \;=\; \frac{1}{\pi\sqrt{e}\,\Delta(S)},
\]
where $C$ is the unique trace minimizer of
$\bS = \{W \succ 0 : \eta(W)\,W = I\}$.
\end{mainthm}
\end{samepage}
Here $\Delta(S) = \exp\!\Big(\int \log|x|\,d\mu_S(x)\Big)$ 
is the Fuglede--Kadison determinant of $S$, and the 
identity $\Capa(\eta)^{-1/(2n)} = (\sqrt{e}\,\Delta(S))^{-1}$ 
follows from~\cite{mai_speicher2024}.

\medskip
\noindent\textbf{Corollary for unsplittable pencils.}


\begin{mainthm}[B]
If $A$ is unsplittable, then $S$ is non-singular and
$f(0) = \frac{1}{\pi}\,\tr(C)$, where $C$ is the
\emph{unique} positive definite solution of $\eta(C) = C^{-1}$.
\end{mainthm}

 \medskip
 \noindent\textbf{Hierarchy.}
 For nonzero Hermitian pencils, the relevant properties 
 form a strict hierarchy:
 \[
 \underbrace{\text{unsplittable}}_{\text{Thm~B}} 
 \;\Longrightarrow\; 
 \underbrace{\text{LR-semisimple}}_{\text{Thm~A}} 
 \;\Longrightarrow\; 
 \underbrace{\text{full}}_{\text{\cite{msy2023}: no atom at $0$}}
 \]

Together, these results give a complete picture.  LR-semisimplicity 
is the precise algebraic condition governing non-singularity of the matrix semicircle at $t = 0$ and ensuring that its covariance map is scalable to the doubly-stochastic map. If the pencil is in addition unsplittable then the scaling is unique up to unitary transformations. If the pencil is not LR-semisimple but still full, then we lose the bounded density but the spectral distribution does not have an atom at $x = 0$ (see the cusp singularity in Example~\ref{exa:singular_density}). 

The formal versions of these theorems are proved as follows: 
Theorem~A as Theorem~\ref{theo_main_general} in 
\S~\ref{section_proof_Thm_A} and  
Theorem~B as Corollary~\ref{theo_main_unsplittable} in 
\S~\ref{section_corollaries}.


\medskip
\noindent\textbf{Strategy of proof.}
The proof consists of three components.

\bigskip
\textbf{Algebraic (\S3):} (i) $\Leftrightarrow$ (iii). \\
LR-semisimple $\Longleftrightarrow$ symmetrically DS-scalable.

\medskip\noindent
$\Rightarrow$:\quad The pencil--map dictionary (\S2.4) reduces
an LR-semisimple pencil to indecomposable blocks. Gurvits'
capacity theory gives DS-scalability of each block. Since
$\eta_A$ is self-adjoint, DS-scalability produces a 2-cycle
$(C, D)$ of the map $F(X) = \eta(X)^{-1}$. We show that $F$
acts as a geodesic reflection in the Riemannian manifold of
positive definite matrices, so that the geometric mean
$C \mathbin{\#} D$ is a fixed point of $F$. This yields
symmetric DS-scalability.

\medskip\noindent
$\Leftarrow$:\quad A splitting lemma for doubly stochastic maps
shows that every splitting pair automatically completes to a
full block-diagonal decomposition. The result then follows
by induction on the matrix size.

\bigskip
\textbf{Forward analytic (\S4):} (iii) $\Rightarrow $ (ii). \\
sym DS-scalable $\Longrightarrow$ non-singular.

\medskip\noindent
We study the HFS form of Speicher's equation,
$\eta(W)W + uW = I$, for the modified Cauchy transform
$W(u) = iG_S(iu)$ with $\Re\, W(u) \succ 0$ for
$\Re\, u > 0$. We apply the Lyapunov--Schmidt reduction at
a trace-minimizing solution of $\eta(W)W = I$. The key step
is establishing that the Jacobian of the reduced bifurcation
equations is positive definite, which follows from a
multiplicative domain identity. The implicit function theorem
then produces a real-analytic extension of $W(u)$ through
$u = 0$, yielding non-singularity and the exact density
formula $f(0) = \frac{1}{\pi}\,\mathrm{tr}(C)$.

\bigskip
\textbf{Converse analytic (\S5):} (ii) $\Rightarrow$  (iii). \\
non-singular $\Longrightarrow$ sym DS-scalable.

\medskip\noindent
A symmetry of Speicher's equation gives
$W(\bar{u}) = W(u)^*$, so $W$ is Hermitian on the real axis.
Non-singularity provides a boundary limit $W_0 = W(0^+)$,
which is therefore Hermitian, positive semidefinite (as a
limit of accretive matrices), and satisfies
$\eta(W_0)\,W_0 = I$. This forces $W_0 \succ 0$, giving
symmetric DS-scalability.

\medskip
\noindent\textbf{Organization.}\;
Section~\ref{sec:definitions} collects the definitions and 
structural results used throughout: matrix semicircles and 
their Cauchy transforms, the hierarchy of matrix pencil 
properties (unsplittable, LR-semisimple, full), completely 
positive maps and their scalings, and the pencil--map 
dictionary relating the two sides. 
Section~\ref{sec:algebraic} establishes the algebraic 
equivalence LR-semisimple $\Leftrightarrow$ symmetrically 
DS-scalable. Section~\ref{sec:analytic_engine} develops the 
Lyapunov--Schmidt analysis that gives the forward analytic 
implication sym DS-scalable $\Rightarrow$ non-singular. 
Section~\ref{sec:main} proves the main theorem by combining 
these with the converse analytic implication. 
Section~\ref{sec:discussion} discusses open questions, 
including the nature of singularities for non-LR-semisimple 
pencils, biased matrix semicircles, and the non-self-adjoint 
case.

\section{Definitions and structural results} 
\label{sec:definitions}


\subsection{Matrix semicircles and covariance maps} 

Let $(\mathcal{A}, \varphi)$ be a non-commutative probability 
space, where $\varphi$ is a faithful tracial state. Self-adjoint 
elements $s_1, \ldots, s_r \in \mathcal{A}$ are called 
\emph{free standard semicircular variables} if they are freely 
independent and each $s_i$ has the semicircular distribution 
$d\mu(x) = \frac{1}{2\pi}\sqrt{4 - x^2}\,dx$ on $[-2, 2]$.

\begin{defi}
A \emph{matrix semicircular variable} is an element 
$S \in M_n(\mathbb{C}) \otimes \mathcal{A}$ of the form
\begin{equation}
\label{equ_defi_semicircle}
S = \sum_{i=1}^{r} A_i \otimes s_i,
\end{equation}
where $A_1, \ldots, A_r \in M_n(\mathbb{C})$ are Hermitian 
matrices and $s_1, \ldots, s_r \in \mathcal{A}$ are free 
standard semicircular variables.
\end{defi}

We write 
$\mathbb{E} = \mathrm{id}_{M_n(\mathbb{C})} \otimes \varphi 
\colon M_n(\mathbb{C}) \otimes \mathcal{A} \to M_n(\mathbb{C})$ 
for the conditional expectation onto the matrix subalgebra.

\begin{defi}
The \emph{covariance map} of a matrix semicircular variable 
$S = \sum_{i=1}^r A_i \otimes s_i$ is the linear map 
$\eta_S \colon M_n(\mathbb{C}) \to M_n(\mathbb{C})$ defined by
\begin{equation}
\label{defi_covariance}
\eta_S(X) := \mathbb{E}[S X S] 
  = \sum_{i=1}^{r} A_i X A_i,
\end{equation}
where the second equality follows from 
$\varphi(s_i s_j) = \delta_{ij}$ (by freeness and the 
normalization of the semicircular law). Since each $A_i$ 
is Hermitian, $\eta_S$ is a completely positive map satisfying 
$\eta_S = \eta_S^*$ (self-adjoint with respect to the 
Hilbert--Schmidt inner product).
\end{defi}

The spectral distribution of $S$ is encoded by its Cauchy 
transform. For $b \in M_n(\mathbb{C})$ such that 
$b \otimes \mathrm{id}_{\mathcal{A}} - S$ is invertible, 
define
\begin{equation}
\label{defi_Cauchy}
G_S(b) = \mathbb{E}\big[
  (b \otimes \mathrm{id}_{\mathcal{A}} - S)^{-1}\big].
\end{equation}
In particular, for $z \in \mathbb{C}^+ := \{z : \Im z > 0\}$, 
the scalar evaluation $G_S(zI)$ is well-defined. The 
\emph{spectral density} of $S$ (when it exists) is recovered 
from the Stieltjes inversion formula:
\begin{equation}
\label{equ_Stieltjes}
f(x) = \frac{1}{\pi} 
  \lim_{\varepsilon \downarrow 0} 
  \Im\big[-\tr\, G_S\big((x + i\varepsilon)I\big)\big],
\end{equation}
where $\tr$ denotes the normalized trace on $M_n(\mathbb{C})$.

The Cauchy transform can be computed from the covariance map 
via the following result.

\begin{theo}[Speicher's equation]
\label{theo_Cauchy_matrix_semicircle}
Let $S$ be a matrix semicircular variable with covariance 
map~$\eta_S$. Then for all $b \in M_n(\mathbb{C})$ with 
$\Im b \succ 0$, the Cauchy transform $G = G_S(b)$ satisfies
\begin{equation}
\label{equ_Speicher}
b\, G(b) = I + \eta_S\!\big(G(b)\big)\, G(b).
\end{equation}
\end{theo}
For a proof see Theorem 2 in Chapter 9 of \cite{mingo_speicher2017} or Sections~6.3--6.4 in~\cite{speicher2019}. This equation was derived by combining the operator-valued $R$-transform theory of \cite{voiculescu95} and the combinatorial cumulant machinery of \cite{speicher98}. A brief outline was given in \cite{fobs2006}. See also \cite{hfs2007} for existence and uniqueness of the solution.

\medskip

Following \cite{hfs2007}, it is convenient to reformulate 
equation~\eqref{equ_Speicher} on the right half-plane. 
Recall that a matrix $W \in M_n(\mathbb{C})$ is 
\emph{accretive} if $\Re W := (W + W^*)/2 \succeq 0$, and 
\emph{strictly accretive} if $\Re W \succ 0$. The 
substitution $u = -iz$ and $W(u) := iG_S(iu\cdot I)$ 
transforms Speicher's equation at $b = iu \cdot I$ into
\begin{equation}
\label{eq:HMS}
\eta_S\!\big(W(u)\big)\, W(u) + u\, W(u) = I.
\end{equation}
By Theorem~2.1 of \cite{hfs2007}, equation~\eqref{eq:HMS} 
has a unique strictly accretive solution $W(u)$ in the 
domain $\{\Re u > 0\}$. The density at $x = 0$ can be 
expressed in terms of the boundary behavior of $W$:
\begin{equation}
\label{equ_density_W}
f(0) = \frac{1}{\pi}\,\tr\!\big(\Re W(0^+)\big),
\end{equation}
provided the non-tangential limit $W(0^+) := 
\lim_{u \to 0,\, \Re u > 0} W(u)$ exists.


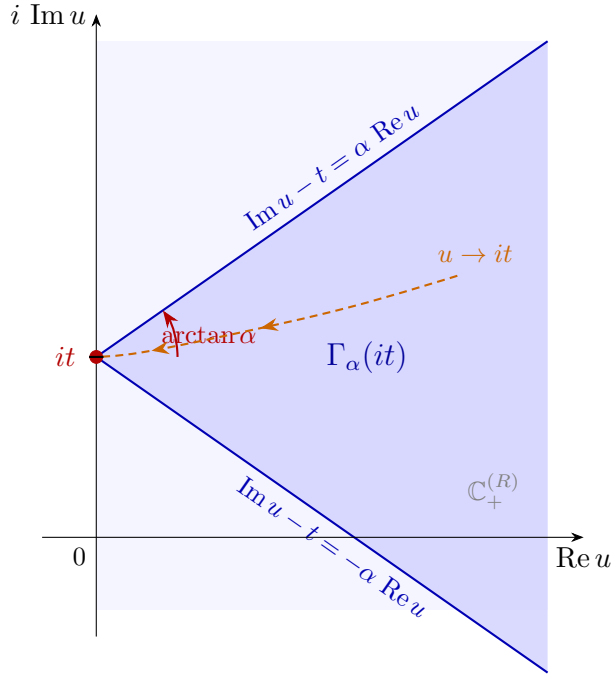
\begin{center}

\begin{tikzpicture}[>=Stealth, scale=1.2]

\def\vt{2.0}       
\def\alph{0.7}     
\def\xmax{5.0}     
\def\ymin{-0.8}
\def\ymax{5.5}

\fill[blue!4] (0,\ymin) rectangle (\xmax,\ymax);

\fill[blue!18, opacity=0.8]
  (0,\vt) -- (\xmax, {\vt+\alph*\xmax}) -- (\xmax, {\vt-\alph*\xmax}) -- cycle;

\draw[blue!70!black, thick]
  (0,\vt) -- (\xmax, {\vt+\alph*\xmax})
  node[pos=0.55, above, sloped, font=\small, text=blue!70!black]
    {$\operatorname{Im} u - t = \alpha\,\operatorname{Re} u$};
\draw[blue!70!black, thick]
  (0,\vt) -- (\xmax, {\vt-\alph*\xmax})
  node[pos=0.55, below, sloped, font=\small, text=blue!70!black]
    {$\operatorname{Im} u - t = -\alpha\,\operatorname{Re} u$};

\draw[->] (-0.6,0) -- (\xmax+0.4,0) node[below] {$\operatorname{Re} u$};
\draw[->] (0,\ymin-0.3) -- (0,\ymax+0.3) node[left] {$i\,\operatorname{Im} u$};

\node[above right, font=\small, text=gray] at (4.0,0.2)
  {$\mathbb{C}_+^{(R)}$};

\filldraw[red!70!black] (0,\vt) circle (2pt)
  node[left=4pt, font=\normalsize] {$it$};

\draw[red!70!black, thick, ->]
  ({0.9},\vt) arc[start angle=0, end angle={atan(\alph)}, radius=0.9];
\node[red!70!black, font=\small] at ({1.25}, {\vt+0.25})
  {$\arctan\alpha$};

\draw[densely dashed, thick, orange!80!black,
      decoration={markings, mark=at position 0.55 with {\arrow{>}},
                            mark=at position 0.85 with {\arrow{>}}},
      postaction={decorate}]
  plot[smooth, tension=0.7] coordinates
    {(4.0,{\vt+0.9}) (2.8,{\vt+0.55}) (1.5,{\vt+0.25}) (0.5,{\vt+0.05}) (0.0,\vt)};
\node[orange!80!black, font=\small] at (4.2, {\vt+1.15}) {$u\to it$};

\node[blue!60!black, font=\large] at (3.0,\vt) {$\Gamma_\alpha(it)$};

\draw[thick] (-0.08,\vt) -- (0.08,\vt);

\node[below left, font=\small] at (0,0) {$0$};

\end{tikzpicture}
\end{center}


\def\Hrp{\mathbb{C}_+^{(R)}}
\begin{defi}[Stolz cone in the right half-plane]
Let $\Hrp:=\{u\in\mathbb{C}:\Re u>0\}$. 
For \(t\in\mathbb{R}\) and \(\alpha>0\), the (non-tangential) Stolz cone with vertex at \(it\) and aperture \(\alpha\) is
\[
\Gamma_\alpha(it)
\;:=\;
\bigl\{\,u\in\Hrp:\; |\Im u - t| \le \alpha\,\Re u \,\bigr\}.
\]
We say \(f:\Hrp\to\mathbb{C}^{m\times m}\) has a \emph{nontangential limit} \(L\) at \(it\) if
\[
\lim_{\substack{u\to it\\ u\in\Gamma_\alpha(it)}} f(u)=L
\quad\text{for some (equivalently, for all) }\alpha>0.
\]
\end{defi}


\begin{defi}[Non-singularity]
\label{defi_nonsingularity}
Let $S\in M_n(\mathbb C)\otimes\mathcal A$ be a self-adjoint matrix semicircular variable 
with Cauchy transform $G_S(z)$ for $\Im z > 0$, and let $W(u) := iG_S(iu)$ 
for $\Re u > 0$.

We say $S$ is \emph{non-singular at $t_0 \in \mathbb{R}$}
if there exist $\delta > 0$ and $\alpha > 0$ such that for every $t$ with 
$|t - t_0| < \delta$:
\begin{enumerate}
\item[\textup{(i)}] $W$ is bounded on $\Gamma_\alpha(it) \cap \{|u - it| < \delta\}$, and
\item[\textup{(ii)}] the non-tangential limit 
$W(it + 0^+) := \displaystyle\lim_{\substack{u \to it \\ u \in \Gamma_\alpha(it)}} W(u)$ 
exists,
\end{enumerate}
and the boundary function $t \mapsto W(it + 0^+)$ is continuous on $(t_0 - \delta,\, t_0 + \delta)$.

We say $S$ is \emph{non-singular} if it is non-singular at every $t_0 \in \mathbb{R}$.
Otherwise, $S$ is \emph{singular}.
\end{defi}

\begin{remark}
Under the substitution $z = iu$, the spectral point $z = x_0$ on the real 
axis corresponds to $u = it_0$ with $t_0 = -x_0$. Non-singularity at 
$t_0 = 0$ therefore governs the behavior of the spectral density near 
$x_0 = 0$, which is the case relevant to the main results of this paper.
\end{remark}

\medskip
\begin{propo}[Algebraicity of the Cauchy transform]
\label{propo_algebraicity}
Let $S = \sum_{i=1}^r A_i \otimes s_i$ be a matrix semicircular 
variable of size~$n$ with covariance map~$\eta_S$. Then the 
scalar Cauchy transform 
$H(z) := \tr\, G_S(zI)$ is an algebraic function of~$z$. 

In particular, the spectral measure~$\mu_S$ decomposes as 
$\mu_S = f(x)\,dx + \sum_{j=1}^{N} m_j\,\delta_{x_j}$ with 
$N < \infty$, where the density~$f$ of the absolutely continuous 
part is real-analytic on 
$\mathbb{R} \setminus \{x_1, \ldots, x_N\}$ except at 
finitely many additional points where it has algebraic 
singularities. The atoms $x_j$ correspond to poles of~$H$ on 
the real axis.
\end{propo}

\begin{proof}
At $b = zI$, Speicher's equation~\eqref{equ_Speicher} reads
\[
z\,G(z) = I + \eta_S\!\big(G(z)\big)\,G(z).
\]
Writing $G(z) = (g_{pq})_{p,q=1}^n$ and expanding $\eta_S$ in 
coordinates, this becomes a system of~$n^2$ polynomial equations 
of degree~$2$ in the~$n^2$ unknowns~$g_{pq}$, with coefficients 
polynomial in~$z$. By the elimination theorem for polynomial 
ideals, each entry~$g_{pq}$ satisfies a nontrivial polynomial 
relation $Q_{pq}(z, g_{pq}) = 0$ and is therefore algebraic 
over~$\mathbb{C}(z)$. Since algebraic functions form a field, 
$H(z) = \tr\,G(z) = \frac{1}{n}\sum_p g_{pp}(z)$ is algebraic. 

As an algebraic function, $H$ has at most finitely many 
singularities on~$\mathbb{R}$, each of which is either a pole 
or a branch point. Poles of~$H$ on the real axis correspond to 
atoms of~$\mu_S$ (with mass given by the residue), and there 
are at most finitely many. Away from these poles and branch 
points, the Stieltjes inversion 
formula~\eqref{equ_Stieltjes} gives a real-analytic density.
\end{proof}

\begin{remark}
\label{rem_algebraicity_context}
If the pencil $A = \sum A_i x_i$ is full (equivalently, 
$\eta_A$ is rank non-decreasing; see 
Theorem~\ref{theo_pencil_map_dictionary}(i)), then $\mu_S$ 
has no atom at zero; more precisely, the mass of the atom at 
zero equals $(n - \mathrm{rank}(A))/n$, where 
$\mathrm{rank}(A)$ is the noncommutative inner rank 
\citep{hoffmann_mai_speicher2024}. The algebraic framework 
connecting inner rank computations to operator-valued free 
probability via the free skew field is developed 
in~\citet{msy2023}. In the full case, the 
spectral measure is therefore purely absolutely continuous, 
$\mu_S = f(x)\,dx$, with~$f$ real-analytic away from finitely 
many algebraic singularities.

Under the stronger assumption that~$\eta_S$ is strictly 
positive (Def.~\ref{defi_indecomposable}(3)), 
\citet{aek2020} prove that the only singularities of~$f$ are 
square-root edges and cubic-root cusps. For polynomials in 
freely independent semicircular variables --- which reduce to 
matrix semicircles via the linearization trick --- algebraicity 
of the Cauchy transform was established by 
\citet{shlyakhtenko_skoufranis2015} using $L^2$-invariants 
and formal power series in noncommuting variables, and extended 
to arbitrary algebraic distributions by \citet{anderson2014} (unpublished).

The main results of this paper address a different question: 
for which matrix semicircles is the density real-analytic in a 
neighborhood of~$0$, when does it satisfy $f(0) > 0$, and what 
controls its value quantitatively?
\end{remark}
The reason we focus on $x = 0$ is the following localization 
result. It relies on an $L^2$ bound on the Cauchy transform 
due to Alt, Erd\H{o}s, and Kr\"uger.

\begin{proposition}[{$L^2$ bound on the Cauchy transform; 
\cite[Prop.~2.1]{aek2020}}]
\label{prop:L2_bound}
Let $S$ be a matrix semicircle with Cauchy transform 
$G_S(z)$ for $z \in \bC^+$. Then
\[
\|G_S(z)\|_2 \;\leq\; \frac{2}{|z|}
\qquad \text{for all } z \in \bC^+,
\]
where $\|X\|_2 := \bigl(\tr(X^* X)\bigr)^{1/2} 
= \frac{1}{\sqrt{n}}\|X\|_F$ is the normalized 
Hilbert--Schmidt norm.
\end{proposition}

\begin{proposition}[Singularity localization]
\label{prop:singularity_localization}
Let $S$ be a matrix semicircle with spectral density $f$. Then
\[
f(x)\le \frac{2}{\pi|x|}
\qquad\text{for all }x\in \mathbb R\setminus(F\cup\{0\}),
\]
where $F$ is the finite exceptional set from Proposition~\ref{propo_algebraicity}.
In particular, $f$ can have an unbounded singularity only at $x=0$.
\end{proposition}

\begin{proof}
Let $g(z):=\tr G_S(z)$. For $z=x+i\varepsilon$ with $x\neq 0$ and $\varepsilon>0$,
\[
|g(z)|
\le \|I\|_2\,\|G_S(z)\|_2
= \|G_S(z)\|_2
\le \frac{2}{|z|}.
\]
If $x\notin F$, then $f$ is real-analytic near $x$, hence
\[
f(x)=\frac1\pi\lim_{\varepsilon\downarrow0}\Im[-g(x+i\varepsilon)].
\]
Therefore
\[
f(x)\le \frac1\pi\limsup_{\varepsilon\downarrow0}|g(x+i\varepsilon)|
\le \frac{2}{\pi|x|}.
\]

Now let $x_0\in F\setminus\{0\}$. Since $F$ is finite, there exists a neighborhood
$U$ of $x_0$ such that $U\cap F=\{x_0\}$ and $0\notin U$. For every
$x\in U\setminus\{x_0\}$ we have already proved
\[
f(x)\le \frac{2}{\pi|x|},
\]
and the right-hand side is bounded on $U$. Hence $f$ is bounded on
$U\setminus\{x_0\}$, so $x_0$ is not an unbounded singularity.
Thus an unbounded singularity can occur only at $0$.
\end{proof}


\subsection{Matrix pencils: Splittability and LR-semisimplicity} 

%
A \emph{(linear) matrix pencil} of size $n$ in $r$ variables is 
a formal expression
\[
  A \;=\; \sum_{i=1}^{r} A_i\, x_i, 
  \qquad A_i \in M_n(\mathbb{C}),
\]
where $x_1, \ldots, x_r$ are formal indeterminates.  
A pencil is \emph{nonzero} if at least one $A_i \neq 0$.
We call $A$ 
\emph{Hermitian} if every coefficient matrix $A_i$ is Hermitian. 
Two pencils $A$ and $B$ of the same size are 
\emph{left--right (LR) equivalent}, written $A \sim_{LR} B$, 
if there exist invertible matrices $P, Q \in GL_n(\mathbb{C})$ 
such that $B_i = P A_i Q$ for all $i$.

The matrix pencil $A$ determines a matrix semicircular variable 
$S = \sum_{i=1}^r A_i \otimes s_i$ and a completely positive 
map $\eta_A(X) = \sum_{i=1}^r A_i X A_i^*$. 
When $A$ is Hermitian, $\eta_A$ is self-adjoint 
($\eta_A^* = \eta_A$), a property that will be essential 
in the main results.
Both $S$ and $\eta_A$ depend only 
on the coefficient matrices $A_1, \ldots, A_r$; in particular, 
commutativity or non-commutativity of the formal variables 
$x_i$ plays no role, and we will not distinguish between the 
two.

\begin{defi}\label{defi:full}
A matrix pencil $A = \sum_{i = 1}^r A_i x_i$ of size $n$ is \emph{full} 
if there are no subspaces $\RR, \LL \subset \mathbb{C}^n$ with 
$A_i(\RR) \subseteq \LL$ for all $i$ and $\dim \RR > \dim \LL$.
\end{defi}

\begin{remark}
This coincides with the notion of full inner rank 
(equal to $n$) for the matrix $A$ viewed as an element 
of $M_n(\mathbb{C}\langle x_1, \ldots, x_r\rangle)$ in 
the sense of \cite{cohn85, cohn95}. For linear pencils, the 
inner rank over the free algebra equals the inner rank 
over the commutative polynomial ring, so the distinction 
is immaterial.
\end{remark}

\begin{defi}\label{defi:splittable}
A matrix pencil $A = \sum_{i = 1}^r A_i x_i$ of size $n$ is 
\emph{splittable} if there exist subspaces $\RR, \LL \subset 
\mathbb{C}^n$ with $\dim \RR = \dim \LL = s$, $1 \leq s \leq 
n-1$, such that $A_i(\RR) \subseteq \LL$ for all $i$. 
A pencil is \emph{unsplittable} if it is nonzero and not splittable.
\end{defi}

Note that a $1 \times 1$ nonzero pencil is automatically 
unsplittable, since there are no intermediate dimensions 
$1 \leq s \leq 0$.

\begin{defi}\label{defi:LR-semisimple}
An $n \times n$ matrix pencil $A = \sum_{i=1}^r A_i x_i$ is 
\emph{LR-semisimple} if there exist invertible 
$L, R \in GL_n(\mathbb{C})$ such that
\[
L A_i R \;=\;
\begin{bmatrix}
B^{(1)}_i & & \\
& \ddots & \\
& & B^{(k)}_i
\end{bmatrix}
\qquad \text{for all } i = 1, \ldots, r,
\]
where $B^{(j)}_i \in M_{n_j}(\mathbb{C})$, 
$\,n_1 + \cdots + n_k = n$, and each sub-pencil 
$B^{(j)} = \sum_{i=1}^r B^{(j)}_i x_i$ is unsplittable.
If this can be achieved with $R = L^\ast$, we say $A$ is 
\emph{C-semisimple}.
\end{defi}

Note that for Hermitian pencils, congruence equivalence 
($X \sim_C Y \iff Y = C X C^\ast$ for $C \in GL_n(\C)$) 
preserves Hermiticity while general LR-equivalence does not.

\begin{remark}\label{rem:hierarchy}
These notions form a strict hierarchy: for nonzero pencils,
\[
 \text{unsplittable} \Longrightarrow  \text{C-semisimple} 
  \Longrightarrow \text{LR-semisimple} 
  \Longrightarrow  \text{full}.
\]
(C-semisimplicity does not appear in the main results but is included for completeness of the hierarchy.)

The first and second implications are immediate from the definitions. 
For the third, we first observe that every nonzero 
unsplittable pencil is full. Indeed, suppose 
$A_i(\RR) \subseteq \LL$ for all $i$ with 
$r = \dim \RR > \ell = \dim \LL$. If $\ell \geq 1$, choose 
any $\ell$-dimensional subspace $\RR' \subseteq \RR$; 
then $A_i(\RR') \subseteq \LL$ with 
$\dim \RR' = \dim \LL = \ell$, and $1 \leq \ell \leq n - 1$ 
(since $\ell < r \leq n$), so $A$ is splittable. 
If $\ell = 0$, then all $A_i$ vanish on $\RR$; 
in particular $n \geq 2$ (for $n = 1$ this would force 
all $A_i = 0$), so choosing any $1$-dimensional 
$\RR' \subseteq \RR$ and any $1$-dimensional 
$\LL' \subseteq \mathbb{C}^n$ gives a splitting pair. 
Since fullness is invariant under LR-equivalence and 
inner rank is additive for block-diagonal pencils 
\cite{cohn85}, LR-semisimplicity implies fullness. 
Both implications are strict, even for $2 \times 2$ 
Hermitian pencils; see 
Examples~\ref{exa:LR_not_C} and~\ref{exa:full_not_LR} 
below.
The pencil hierarchy will later be shown to correspond 
to a parallel hierarchy of properties of the covariance 
map $\eta_A$; see Section \ref{sec:pencil_map_equivalences}.
\end{remark}

\begin{exa}\label{exa:LR_not_C}
The Hermitian pencil
\begin{align*}
A &=\; \begin{pmatrix} 0 & x_2 + (1+i)\,x_1 \\ 
  x_2 + (1-i)\,x_1 & 0 \end{pmatrix},
  \\
&A_1 = \begin{pmatrix} 0 & 1+i \\ 1-i & 0 \end{pmatrix},\quad
A_2 = \begin{pmatrix} 0 & 1 \\ 1 & 0 \end{pmatrix},
\end{align*}
is LR-semisimple but not C-semisimple.

\medskip\noindent
\textbf{LR-semisimple.}\; 
Take $L = I$ and $R = \bigl(\begin{smallmatrix} 0 & 1 \\ 
1 & 0 \end{smallmatrix}\bigr)$. Then
\[
L A_1 R = \begin{pmatrix} 1+i & 0 \\ 0 & 1-i \end{pmatrix},
\qquad
L A_2 R = \begin{pmatrix} 1 & 0 \\ 0 & 1 \end{pmatrix},
\]
so $A$ is LR-equivalent to a direct sum of two 
$1 \times 1$ (unsplittable) pencils.

\medskip\noindent
\textbf{Not C-semisimple.}\;
Suppose $U = \bigl(\begin{smallmatrix} a & b \\ 
c & d \end{smallmatrix}\bigr) \in GL_2(\mathbb{C})$ 
simultaneously diagonalizes $UA_iU^*$ for $i = 1, 2$. 
The off-diagonal entries are:
\begin{align*}
(UA_2 U^*)_{12} &= b\bar{c} + a\bar{d}, \\
(UA_1 U^*)_{12} &= b(1{-}i)\bar{c} + a(1{+}i)\bar{d}.
\end{align*}
Setting both to zero and substituting 
$a\bar{d} = -b\bar{c}$ from the first equation into 
the second:
\[
b(1{-}i)\bar{c} - (1{+}i)\,b\bar{c} 
  = -2i\,b\bar{c} = 0.
\]
Hence $b\bar{c} = 0$, which forces $a\bar{d} = 0$. 
But then $\det U = ad - bc = 0$, 
contradicting the invertibility of $U$.
\end{exa}


\begin{exa}\label{exa:full_not_LR}
The Hermitian pencil
\[
A = \begin{pmatrix} 0 & x_1 \\ x_1 & x_2 \end{pmatrix},
\qquad
A_1 = \begin{pmatrix} 0 & 1 \\ 1 & 0 \end{pmatrix},\quad
A_2 = \begin{pmatrix} 0 & 0 \\ 0 & 1 \end{pmatrix},
\]
is full but not LR-semisimple.

\medskip\noindent
\textbf{Full.}\;
If $A_i(\RR) \subseteq \LL$ for all $i$, then in 
particular $A_1(\RR) \subseteq \LL$, so 
$\dim\LL \geq \dim A_1(\RR) = \dim\RR$ since $A_1$ 
is invertible. Hence $A$ is full.

\medskip\noindent
\textbf{Splittable.}\;
$A_i(\operatorname{span}(e_1)) \subseteq 
\operatorname{span}(e_2)$ for both $i$.

\medskip\noindent
\textbf{Not LR-semisimple.}\;
Since $A$ is $2 \times 2$ and splittable, 
LR-semisimplicity would mean that $A$ is LR-equivalent 
to a diagonal pencil: there exist invertible 
$L, R \in GL_2(\mathbb{C})$ such that $LA_iR$ is 
diagonal for every $i$.

Since $\operatorname{rank}(A_2) = 1$, the diagonal 
matrix $LA_2R$ also has rank $1$, so exactly one 
diagonal entry is nonzero. Swapping columns of $R$ 
and rows of $L$ if necessary, we may assume 
$LA_2R = \operatorname{diag}(0, \lambda)$ with 
$\lambda \neq 0$.

Write $R = (R_1 \mid R_2)$ in columns. The first 
column of $LA_2R$ is $LA_2R_1 = 0$, so 
$A_2 R_1 = 0$, i.e., 
$R_1 \in \ker A_2 = \operatorname{span}(e_1)$. 
Hence $R_1 = a\,e_1$ for some $a \neq 0$, and we 
can write 
$R = \bigl(\begin{smallmatrix} a & b \\ 0 & d 
\end{smallmatrix}\bigr)$ with $a, d \neq 0$.

Now $LA_1R$ must also be diagonal. In particular, its 
$(2,1)$-entry vanishes:
\[
(LA_1R)_{21} = (L A_1 R_1)_2 = 0.
\]
But $A_1 R_1 = a\,A_1 e_1 = a\,e_2$, so $L$ must 
send $e_2$ to a vector with zero second component. 
On the other hand, the $(1,2)$-entry of $LA_2R$ 
vanishes, giving $(LA_2 R_2)_1 = 0$. Since 
$A_2 R_2 = d\,e_2$, this says $L$ sends $e_2$ to a 
vector with zero first component. Together, 
$Le_2 = 0$, contradicting the invertibility of $L$.
\end{exa}


\subsection{Completely positive maps: indecomposability, capacity, DS-scalability} 

Recall that a matrix $B \in M_n(\bC)$ is \emph{positive semidefinite}, $B \succeq 0$,  if $(x, B x) \geq 0$ for all $x \in \bC^n$, and \emph{positive definite}, $B \succ 0$,  if $(x, B x) > 0$ for all non-zero $x \in \bC^n$. 

\begin{defi}
\label{defi_CP}
A linear map $\eta: M_n(\bC) \to M_n(\bC)$ is said to be
\begin{enumerate}
\item \emph{positive} if $\eta(B) \succeq 0$ for all $B \succeq 0$;
\item \emph{completely positive (CP)} if all amplifications $\eta^{(k)}: M_k(\bC) \otimes M_n(\bC) \to M_k(\bC) \otimes M_n(\bC)$, defined by $\eta^{(k)} := \id_{M_k(\bC)} \otimes \eta$ for $k \in \N$, are positive. 
\end{enumerate}
\end{defi}

By the Choi--Kraus theorem (see, e.g., \cite[Theorem~2.22]{watrous2018}), for every CP map $\eta$ there exist matrices $A_1, \ldots, A_r \in M_n(\bC)$ such that 
\[
\eta(X) = \sum_{i=1}^r A_i X A_i^\ast.
\]
The matrices $A_1, \ldots, A_r$ are called \emph{Kraus operators} for $\eta$. The representation is not unique.

\begin{defi}[Adjoint]
\label{defi_adjoint}
The \emph{adjoint} of a linear map $\eta: M_n(\bC) \to M_n(\bC)$ with respect to the inner product $\langle X, Y\rangle = \Tr(X Y^\ast)$ is the unique linear map $\eta^\ast: M_n(\bC) \to M_n(\bC)$ satisfying
\[
\langle \eta^\ast(X), Y \rangle = \langle X, \eta(Y) \rangle \quad \text{for all } X, Y \in M_n(\bC).
\]
A linear map $\eta$ is called \emph{self-adjoint} (or \emph{Hermitian}) if $\eta^\ast = \eta$.
\end{defi}

If $\eta$ is a CP map with Kraus representation $\eta(X) = \sum_{i=1}^r A_i X A_i^\ast$, then $\eta^\ast(X) = \sum_{i=1}^r A_i^\ast X A_i$. Indeed, for all $X, Y$,
\begin{align*}
\langle X, \eta(Y) \rangle &= \sum_{i=1}^r \Tr(X (A_i Y A_i^\ast)^\ast) = \sum_{i=1}^r \Tr(X A_i Y^\ast A_i^\ast) = \sum_{i=1}^r \Tr(A_i^\ast X A_i \cdot Y^\ast) 
\\
&= \langle \textstyle\sum_i A_i^\ast X A_i,\, Y\rangle.
\end{align*}
In particular, $\eta^\ast$ is again completely positive.

\begin{lemma}
\label{lemma_hermitian_kraus}
A CP map $\eta$ is self-adjoint if and only if it admits a Kraus representation with Hermitian operators: $\eta(X) = \sum_{j} H_j X H_j$ with $H_j^\ast = H_j$ for every $j$.
\end{lemma}
\begin{proof}
If $\eta(X) = \sum_j H_j X H_j$ with $H_j^\ast = H_j$, then $\eta^\ast(X) = \sum_j H_j^\ast X H_j = \sum_j H_j X H_j = \eta(X)$.

Conversely, suppose $\eta^\ast = \eta$ and write $\eta(X) = \sum_k A_k X A_k^\ast$. Decompose each Kraus operator as $A_k = H_k + i S_k$ where $H_k = \frac{1}{2}(A_k + A_k^\ast)$ and $S_k = \frac{1}{2i}(A_k - A_k^\ast)$ are Hermitian. Then
\[
A_k X A_k^\ast = H_k X H_k + S_k X S_k + i(S_k X H_k - H_k X S_k).
\]
Summing over $k$ and using the Kraus formula for the adjoint,
\[
\eta(X) - \eta^\ast(X) = 2i \sum_k (S_k X H_k - H_k X S_k) = 0,
\]
so $\eta(X) = \sum_k (H_k X H_k + S_k X S_k)$, which is a Hermitian Kraus representation.
\end{proof}

In particular, every CP map arising from a Hermitian matrix pencil $A = \sum A_i x_i$ with $A_i^\ast = A_i$ is self-adjoint.

\begin{defi} 
\label{defi_indecomposable}
A CP linear map $\eta:  M_n(\bC) \to M_n(\bC)$ is said to be 
\begin{enumerate}
\item \emph{rank non-decreasing (RND)} if $\rank(\eta(B)) \geq \rank(B)$ for all $B \succeq 0$;
\item \emph{indecomposable} if $\rank(\eta(B)) > \rank(B)$ for all $B \succeq 0$ with $1 \leq \rank(B) < n$;
\item \emph{strictly positive} if there exists $\alpha > 0$ such that $\eta(B) \succeq \alpha \Tr(B) I_n$ for all $B \succeq 0$. 
\end{enumerate}
\end{defi}

We have the strict inclusions:
\[
\text{strictly positive} \;\subsetneq\; \text{indecomposable} \;\subsetneq\; \text{rank non-decreasing}.  
\]

\begin{exa}[RND $\nRightarrow$ indecomposable]
\label{exa_RND_not_indecomp}
For $n = 2$, the identity map $\eta(B) = B$ is rank non-decreasing but not indecomposable, since $\rank(\eta(B)) = \rank(B)$ for every rank-$1$ positive semidefinite matrix $B$.
\end{exa}

\begin{exa}[Indecomposable $\nRightarrow$ strictly positive]
\label{exa_indecomposable_not_SP}
Let $n = 3$ and let $E_{ij}$ denote the matrix units. Define
\[
\Phi(B) = \sum_{(i, j) \ne (3,3)} E_{ij} B E_{ji} =  \sum_{i, j = 1}^3 E_{ij} B E_{ji} - E_{33} B E_{33}.
\]
This is completely positive, and can also be written as
\[
\Phi(B) = \Tr(B) I_3 - B_{33} E_{33} 
= \bmatr{ \Tr(B) & 0 & 0 \\ 0 & \Tr(B) & 0 \\ 0 & 0 & \Tr(B) - B_{33}}.
\]
We show that $\Phi$ is indecomposable. If $\rank(B) = 1$, write $B = v v^\ast$ with $\|v\| = 1$. Then 
\[
\Phi(v v^\ast) = \diag(1, 1, 1 - |v_3|^2), 
\]
so $\rank(\Phi(B)) \geq 2 > \rank(B)$. 

If $\rank(B) = 2$, then $B_{11} + B_{22} > 0$ (otherwise $B \succeq 0$ would force the first two rows and columns to vanish, giving $\rank(B) \leq 1$). Thus 
\[
\Phi(B) = \diag(\Tr(B), \Tr(B), B_{11} + B_{22}) > 0,
\]
so $\rank \Phi(B) = 3 > \rank(B)$. 
Therefore $\Phi$ is indecomposable. However, for $B = e_3 e_3^\ast$,
\[
\Phi(B) = \diag(1, 1, 0) \nsucceq \alpha I_3 \quad \text{for any } \alpha > 0,
\]
so $\Phi$ is not strictly positive.
\end{exa}


\begin{defi}
\label{defi_DS}
A CP linear map $\eta:  M_n(\bC) \to M_n(\bC)$ is said to be  \emph{doubly stochastic (DS)} if $\eta(I_n) = I_n$ and $\eta^\ast(I_n) = I_n$. The first condition means $\eta$ is \emph{unital}; the second means $\eta$ is \emph{trace-preserving}.
\end{defi}

\begin{defi}[Operator scaling]
\label{defi_operator_scaling}
Let $\eta$ be a CP map on $M_n(\bC)$ and let $c_1, c_2 \in M_n(\bC)$. The \emph{operator scaling} $S_{c_1, c_2}(\eta)$ is defined by
\[
S_{c_1, c_2}(\eta)(X) := c_1 \eta(c_2^\ast X c_2) c_1^\ast.
\]
We sometimes write $\eta_{c_1, c_2}$ for $S_{c_1, c_2}(\eta)$.
\end{defi}

\begin{defi}[Operator Sinkhorn Iteration (OSI)]
\label{defi_OSI}
Define the \emph{row normalization} and \emph{column normalization} of a CP map $\eta$ (assuming the relevant inverses exist) by
\begin{align*}
R(\eta) &:= S_{\eta(I)^{-1/2},\, I}(\eta), \\
C(\eta) &:= S_{I,\, \eta^\ast(I)^{-1/2}}(\eta).
\end{align*}
The \emph{Operator Sinkhorn Iteration} (OSI) is the iterative process $\ldots C R C R(\eta)$. 
\end{defi}

The distance of a CP map from the doubly stochastic class can be measured by
\begin{equation}
\label{equ_DS_distance}
\DS(\eta) := \Tr \bigl[(\eta(I) - I)^2\bigr] + \Tr \bigl[(\eta^\ast(I) - I)^2\bigr].
\end{equation}

\medskip
An important scaling invariant of CP maps is the capacity. 

 
\begin{defi}[Capacity]
\label{defi_capacity}
For a CP linear map $\eta: M_n(\bC) \to M_n(\bC)$, its \emph{capacity} is
\begin{equation}
\label{equ_defi_capacity}
\Capa(\eta) := \inf \{\det(\eta(B)) \mid B \succ 0, \; \det(B) = 1\}.
\end{equation}
\end{defi}


\begin{propo}[Scaling formula for capacity; Prop.~2.7 in \cite{ggow2020}]
\label{propo_scalings_capacity}
Let $\eta$ be a completely positive map and let $c_1$, $c_2$ be invertible matrices. Then 
\[
\Capa(\eta_{c_1,c_2}) = \det(c_1)^2 \det(c_2)^2 \Capa(\eta).
\]
\end{propo}


\begin{lemma}
\label{lemma_DS_unit_capacity}
If $\eta: M_n(\bC) \to M_n(\bC)$ is a doubly stochastic CP map, then $\Capa(\eta) = 1$.
\end{lemma}

\begin{proof}
The upper bound $\Capa(\eta) \leq \det(\eta(I)) = 1$ is immediate from the definition.

For the lower bound, let $B \succ 0$ with $\det(B) = 1$ and set
$h(t) = \log\det(\eta(B^t))$. Write $L = \log B$, so that
$\frac{d}{dt}\eta(B^t) = \eta(B^t L)$ and
$\frac{d^2}{dt^2}\eta(B^t) = \eta(B^t L^2)$.
Define the unital CP map
\[
\Psi_t(X) = \eta(B^t)^{-1/2}\,\eta(B^{t/2}\, X\, B^{t/2})\,\eta(B^t)^{-1/2}.
\]
By Lemma~\ref{lemma_logdet_convexity} 
\[
h''(t) = \Tr[\Psi_t(L^2)] - \Tr[\Psi_t(L)^2] \geq 0.
\]

At $t = 0$: unitality gives $h(0) = \log\det(\eta(I)) = 0$,
and trace-preservation gives
\[
h'(0) = \Tr\bigl[\eta(I)^{-1}\,\eta(L)\bigr] = \Tr(\eta(L)) = \Tr(L) = \log\det(B) = 0.
\]
A convex function with $h(0) = 0$ and $h'(0) = 0$ satisfies
$h(t) \geq 0$ for all $t \geq 0$. In particular,
$\log\det(\eta(B)) = h(1) \geq 0$, i.e., $\det(\eta(B)) \geq 1$.
\end{proof}


\begin{lemma}[Critical points are capacity minimizers]
\label{lemma_exact_fixed_points_capacity}
Let $\eta: M_n(\bC) \to M_n(\bC)$ be a completely positive map and let $C \succ 0$ satisfy
\begin{equation}
\label{equ_capacity_FOC}
\eta^\ast\bigl(\eta(C)^{-1}\bigr) = C^{-1}.
\end{equation}
Then $\Capa(\eta) = \det(\eta(C))/\det(C)$.
\end{lemma}

\begin{proof}
The scaling $Z := S_{\eta(C)^{-1/2},\, C^{1/2}}(\eta)$ satisfies 
\[
Z(I) = \eta(C)^{-1/2}\,\eta(C)\,\eta(C)^{-1/2} = I.
\]
 For the adjoint, $Z^\ast = S_{C^{1/2},\, \eta(C)^{-1/2}}(\eta^\ast)$, so
\[
Z^\ast(I) = C^{1/2}\,\eta^\ast\bigl(\eta(C)^{-1}\bigr)\,C^{1/2} = C^{1/2}\,C^{-1}\,C^{1/2} = I,
\]
where we used~\eqref{equ_capacity_FOC}. Hence $Z$ is doubly stochastic, so by Lemma \ref{lemma_DS_unit_capacity} $\Capa(Z) = 1$. By the scaling formula (Proposition~\ref{propo_scalings_capacity}),
\[
\Capa(\eta) = \frac{\Capa(Z)}{\det(\eta(C)^{-1}) \det(C)} = \frac{\det(\eta(C))}{\det(C)}.
\]
\end{proof}

\begin{remark}
A quantitative generalization to $\varepsilon$-approximate fixed points appears as Lemma~3.9 in~\cite{ggow2020}.
\end{remark}


\begin{defi}[DS-scalability]
\label{defi_DS_scalable}
A CP map $\eta:M_n(\mathbb{C})\to M_n(\mathbb{C})$ is
\emph{DS-scalable} if there exist invertible matrices $c_1, c_2\in GL_n(\mathbb{C})$
such that the scaled map $S_{c_1,\,c_2}(\eta)$ is doubly stochastic.
\end{defi}

\begin{defi}[Symmetric DS-scalability]
\label{defi_symmetric_DS}
A Hermitian CP map $\eta:M_n(\mathbb C)\to M_n(\mathbb C)$ is
\emph{symmetrically DS-scalable} if there exists a positive definite matrix
$C \succ 0$ such that the symmetrically scaled map $S_{C^{1/2},\, C^{1/2}}(\eta)$
is doubly stochastic.
\end{defi}

Symmetric DS-scalability is DS-scalability with the additional constraint
$c_1 = c_2 = C^{1/2} \succ 0$. Since $\eta$ is Hermitian, the scaled map
$\tilde\eta := S_{C^{1/2}, C^{1/2}}(\eta)$ is also Hermitian, so the two DS
conditions $\tilde\eta(I) = I$ and $\tilde\eta^\ast(I) = I$ reduce to a single
equation. Explicitly, $\tilde\eta(I) = C^{1/2}\eta(C)C^{1/2}$, so
$\tilde\eta(I) = I$ is equivalent to
\begin{equation}
\label{equ_symmetric_scalability}
\eta(C) = C^{-1}.
\end{equation}

We will also need the following result, which shows that 
unit-eigenvalue eigenstates of a DS map lie in its 
multiplicative domain. It will play a key role in the 
Lyapunov--Schmidt analysis of~\S\ref{sec:analytic_engine}.

\begin{lemma}
\label{lemma_unit_eigenstates}
Let $\Phi: M_n(\bC) \to M_n(\bC)$ be doubly stochastic CP map. Assume that $A$ is Hermitian and $\Phi(A) = \lambda A$ with $|\lambda| = 1$. Then the Kadison--Schwarz inequality is satisfied with equality,  i.e. $\Phi(A^2) =  \Phi(A)^2$ and $\Phi$ restricts to a 
$\ast$-homomorphism on the $C^\ast$-algebra $C^\ast(A)$ generated by $A$ 
and $I$.
\end{lemma}

\begin{proof}
Since $\Phi$ preserves adjoints, $\Phi(A) = \lambda A$ with $A = A^*$ 
forces $\lambda \in \R$; combined with $|\lambda| = 1$ this gives 
$\lambda = \pm 1$. The Kadison--Schwarz inequality 
(Lemma~\ref{lem:kadison_schwarz}) yields
\[
\Phi(A^2) \geq \Phi(A)^2 = \lambda^2 A^2 = A^2.
\]
Since $\Phi$ is doubly stochastic, it is trace-preserving, so 
$\tr(\Phi(A^2) - A^2) = 0$. A positive semidefinite matrix with 
trace zero vanishes, hence $\Phi(A^2) = A^2 = \Phi(A)^2$. 
By Theorem~\ref{theo_choi}, $\Phi$ restricts to a 
$*$-homomorphism on $C^*(A)$.
\end{proof}


\subsection{Equivalences between pencil structure and map properties} 
\label{sec:pencil_map_equivalences}

The properties of a matrix pencil $A = \sum_{i=1}^d A_i x_i$ 
are faithfully reflected in the properties of the associated CP 
map $\eta_A(B) = \sum_{i=1}^d A_i B A_i^*$.  
The following three equivalences form the ``pencil--map dictionary'' 
that underlies the rest of the paper.

\begin{theo}[Pencil--Map Dictionary]
\label{theo_pencil_map_dictionary}
Let $A = \sum_{i=1}^d A_i x_i$ be an $n \times n$ matrix pencil and 
$\eta_A(B) = \sum_{i=1}^d A_i B A_i^*$ the associated CP map.  Then:
\begin{enumerate}
\item[\textup{(i)}] $A$ is full $\;\Longleftrightarrow\;$ 
      $\eta_A$ is rank non-decreasing $\;\Longleftrightarrow\;$ $\Capa(\eta_A)>0$.
\item[\textup{(ii)}] $A$ is unsplittable $\;\Longleftrightarrow\;$ 
      $\eta_A$ is indecomposable.
\item[\textup{(iii)}] $A$ is LR-semisimple $\;\Longleftrightarrow\;$ 
      there exist $c_1,c_2\in GL_n(\C)$ and a diagonal subalgebra $\DD$ 
      such that $S_{c_1,c_2}(\eta_A)$ is $\DD$-CI.  
\end{enumerate}
\end{theo}
\begin{remark}
If $c_1 = c_2$, $\DD$-CI-scalability characterizes C-semisimplicity.
\end{remark}

\noindent
On the pencil side, the hierarchy is
\[
  \text{unsplittable}  \Longrightarrow \text{C-semisimple}  \Longrightarrow \text{LR-semisimple} 
  \Longrightarrow \text{full},
\]
and correspondingly, on the map side:
\begin{align*}
\text{indecomposable} &
 \Longrightarrow \text{$\DD$-CI--symmetrically scalable}
  \Longrightarrow \text{$\DD$-CI-- scalable} 
  \\
 & \Longrightarrow \text{rank non-decreasing}
  \Longleftrightarrow  \Capa > 0.
\end{align*}
All implications above are strict.

\medskip
Before proving the theorem, we introduce the $\DD$-CI property 
that appears in part~(iii).

\begin{defi}[$\mathcal{D}$-bimodular map]
\label{defi:D-bimodular}
Let $P_1,\dotsc,P_r$ be pairwise orthogonal projections in $M_n(\C)$ 
with $\sum_{j=1}^r P_j = I$, and set 
$\DD := \operatorname{span}\{P_1,\dotsc,P_r\}$.  
A CP map $\eta\colon M_n(\mathbb{C})\to M_n(\mathbb{C})$ is
\emph{$\mathcal{D}$-bimodular} if
\[
  \eta(D_1 X D_2) = D_1\,\eta(X)\,D_2
  \quad\text{for all } D_1,D_2\in\mathcal{D},\ X\in M_n(\mathbb{C}).
\]
\end{defi}

\begin{defi}[$\DD$-CI map]
\label{defi:DCI}
Let $\mathcal{D} = \operatorname{span}\{P_1,\dotsc,P_r\}$ as above.
A completely positive map $\eta\colon M_n(\C)\to M_n(\C)$ is 
\emph{$\DD$-corner-indecomposable} (abbrev.\ \emph{$\DD$-CI}) if:
\begin{enumerate}
\item[\textup{(a)}] \textup{(\textbf{$\DD$-bimodule})}\;
$\eta$ is $\DD$-bimodular, 
\item[\textup{(b)}] \textup{(\textbf{Corner-indecomposable})}\;
  For each~$j$, the CP map 
  $\eta_j := P_j\,\eta(\,\cdot\,)\,P_j$ on $P_j M_n(\C) P_j$ 
  is indecomposable.
\end{enumerate}
\end{defi}

\begin{remark}
The identity map is $\DD$-bimodular for every diagonal 
subalgebra~$\DD$, but it is $\DD$-CI only when each $P_j$ has rank 
one: the corner map $\eta_j(X) = P_j X P_j$ on 
$P_j M_n(\C) P_j \cong M_{n_j}(\C)$ is the identity on~$M_{n_j}(\C)$, 
which is indecomposable if and only if $n_j = 1$.
\end{remark}

We also record a Kraus characterization of $\DD$-CI maps that 
will be used in the proof of part~(iii).

\begin{lemma}[Kraus characterization of $\DD$-CI maps]
\label{lem:kraus_DCI}
A completely positive map $\eta\colon M_n(\C)\to M_n(\C)$ is $\DD$-CI
if and only if it admits a Kraus representation
$\eta(X) = \sum_{i=1}^m A_i X A_i^*$ with each 
$A_i\in\DD' = \bigoplus_{j=1}^r P_j M_n(\C) P_j$, and for each~$j$ 
the corner map
\[
\eta_j(X) = \sum_{i=1}^m (P_j A_i P_j)\,X\,(P_j A_i P_j)^*
\quad\text{on } P_j M_n(\C) P_j
\]
is indecomposable.
\end{lemma}

\begin{proof}
The corner-indecomposability condition is a direct restatement of 
part~(b) of Definition~\ref{defi:DCI}, so it suffices to show that 
$\eta$ is $\DD$-bimodular if and only if it admits Kraus operators 
in~$\DD'$.

$(\Leftarrow)$\; If each $A_i$ commutes with every element of~$\DD$, then 
$\eta(D_1 X D_2) = \sum_i A_i D_1 X D_2 A_i^* = D_1\,\eta(X)\,D_2$.

$(\Rightarrow)$\; Start from an arbitrary Kraus representation 
$\eta(X) = \sum_i B_i X B_i^*$.  $\DD$-bimodularity gives, for every 
$\ell,j,k,m$,
\begin{equation}
\label{eq:bimod_blocks}
P_\ell\,\eta(P_j X P_k)\,P_m 
  = \delta_{\ell j}\,\delta_{mk}\;\eta(P_j X P_k).
\end{equation}
The left-hand side equals 
$\sum_i (P_\ell B_i P_j)\,X\,(P_k B_i^* P_m)$.  Taking $\ell\ne j$:  
for all~$X$ and all $k,m$ we have 
$\sum_i (P_\ell B_i P_j) X (P_k B_i^* P_m) = 0$, which forces 
$P_\ell B_i P_j = 0$ for every~$i$.  Therefore 
$B_i = \sum_j P_j B_i P_j$ for each~$i$, i.e., each $B_i$ already 
lies in~$\DD'$.

Summing~\eqref{eq:bimod_blocks} over $j,k$ with $\ell = j$, $m = k$:
\[
\eta(X) 
  = \sum_{j,k} \eta(P_j X P_k) 
  = \sum_{j,k}\sum_i (P_j B_i P_j)\,X\,(P_k B_i^* P_k) 
  = \sum_i B_i\,X\,B_i^*.
  \qedhere
\]
\end{proof}

In other words, $\eta$ is $\DD$-CI precisely when its Kraus operators 
respect the block structure imposed by~$\DD$, and each diagonal block 
acts indecomposably.

\medskip
We now prove the three parts of Theorem~\ref{theo_pencil_map_dictionary}.

\begin{proof}[Proof of part \textup{(i)}]
The equivalence full $\Leftrightarrow$ rank non-decreasing is 
Proposition~2.10 in~\cite{msy2023}, which refers to 
Theorem~1.4 of~\cite{ggow2020} (equivalence (5)~$\Leftrightarrow$~(7), 
stated there without proof).  We provide the elementary argument in 
Appendix~\ref{proof_full_eq_RND}; here we sketch the key ideas.

Fullness means no subspace pair $(\RR,\LL)$ has positive excess 
$\dim\RR - \dim\LL > 0$ subject to $A_i(\RR)\subseteq\LL$, 
while rank non-decreasing means 
$\dim\im(\eta_A(B)) \ge \dim\im(B)$ for all $B\succeq 0$.

For $(\Leftarrow)$: if $A$ is not full, there exists a pair 
$(\RR,\LL)$ with $A_i(\RR)\subseteq\LL$ and $\dim\RR > \dim\LL$.  
Taking $B = P_\RR$ gives 
$\rank(\eta_A(B)) \le \dim\LL < \dim\RR = \rank(B)$.

For $(\Rightarrow)$: if $\eta_A$ is not rank non-decreasing, 
some $B\succeq 0$ witnesses $\rank(\eta_A(B)) < \rank(B)$.  Setting 
$\RR = \im(B)$ and $\LL = \im(\eta_A(B))$, the same argument as 
in Step~1 of the proof of part~(ii) below shows 
$A_i(\RR)\subseteq\LL$ with $\dim\LL < \dim\RR$, contradicting 
fullness.

The equivalence rank non-decreasing $\Leftrightarrow$ $\Capa(\eta_A) > 0$ 
is due to Gurvits:

\begin{propo}[Lemma 4.5 in \cite{gurvits2004}] 
\label{propo_RN_map_capacity}
A completely positive map $\eta$ is rank non-decreasing if and only 
if $\Capa(\eta) > 0$. 
\end{propo}

The proof is based on the study of mixed discriminants and their 
relation to rank non-decreasing maps and capacity.

\end{proof}


\begin{proof}[Proof of part \textup{(ii)}]

\textit{(Splittable $\Rightarrow$ decomposable).}\;
By splittability, there exist subspaces $\RR,\LL\subset \C^n$ with 
$1\le \dim\RR = \dim\LL = s \le n-1$ and $A_i(\RR)\subseteq\LL$ 
for all $i$.  Choose $B = P_\RR$ (the orthogonal projection onto~$\RR$), 
so $\rank(B) = s$.  For each~$i$, 
$\im(A_i B A_i^*) \subseteq A_i(\RR) \subseteq \LL$, hence
\[
\rank\bigl(\eta_A(B)\bigr) \le \dim\LL = s = \rank(B).
\]
Since $1\le s\le n-1$, this shows $\eta_A$ is not indecomposable.

\medskip
\textit{(Decomposable $\Rightarrow$ splittable).}\;
Since $\eta_A$ is not indecomposable, there exists $B\succeq 0$ with 
$\rank(B) = s$, $1\le s\le n-1$, and $\rank(\eta_A(B))\le s$.  Set
\[
\RR := \im(B), \qquad \LL := \im(\eta_A(B)), \qquad t := \dim\LL \le s.
\]

\textit{Step~1: $A_i(\RR)\subseteq\LL$ for all~$i$.}\;
Since $\eta_A(B)\ge 0$, we have $\LL^\perp = \ker(\eta_A(B))$.  
For any $\xi\in\LL^\perp$,
\[
0 = \langle\xi,\,\eta_A(B)\,\xi\rangle 
  = \sum_i \langle A_i^*\xi,\; B\,A_i^*\xi\rangle.
\]
Each summand is nonnegative, so $\langle A_i^*\xi,\, B A_i^*\xi\rangle = 0$ 
for all~$i$.  Since $B\succeq 0$, this gives $A_i^*\xi\in\ker(B) = \RR^\perp$.  
Hence $A_i^*(\LL^\perp)\subseteq\RR^\perp$, equivalently 
$A_i(\RR)\subseteq\LL$.

\medskip
\textit{Step~2: Constructing the splitting pair.}\;
If $t\ge 1$, choose any subspace $\RR'\subseteq\RR$ with 
$\dim\RR' = t$ and set $\LL' = \LL$.  Then 
$A_i(\RR')\subseteq A_i(\RR)\subseteq\LL'$ for all~$i$, and 
$\dim\RR' = \dim\LL' = t$ with $1\le t\le n-1$.

If $t = 0$, then $A_i(\RR) = \{0\}$ for all~$i$ (every vector 
in~$\RR$ lies in $\ker A_i$).  Since $s\ge 1$, pick any 
one-dimensional $\RR'\subset\RR$ and any one-dimensional 
$\LL'\subset\C^n$; then $A_i(\RR') = \{0\}\subseteq\LL'$ for all~$i$.

In both cases $(\RR',\LL')$ is a nontrivial splitting pair.
\end{proof}

Part~(ii) provides a convenient computational criterion for 
unsplittability: it suffices to check that 
$\rank(\eta_A(B)) > \rank(B)$ for every $B\succeq 0$ of intermediate rank.

\begin{exa}
\label{exa_antisymmetric_2}
Consider the $3\times 3$ pencil with $d = 3$ generators
\[
A = \frac{1}{\sqrt{2}} \bmatr{ 0 & x_1 & x_2 \\ 
  -x_1 & 0 & x_3 \\ -x_2 & -x_3 & 0}.
\]
A direct computation gives
\[
\eta_A(B) = \tfrac{1}{2}\bigl((\Tr B)\, I_3 - B\bigr).
\]
Diagonalizing $B = \diag(\lambda_1,\lambda_2,\lambda_3)$ with 
$\lambda_1\ge\lambda_2\ge\lambda_3\ge 0$:

If $\rank(B) = 1$, then $B = \diag(\lambda_1,0,0)$ with $\lambda_1 > 0$, and
\[
\eta_A(B) = \tfrac{\lambda_1}{2}\,\diag(0,1,1),
\qquad \rank(\eta_A(B)) = 2 > 1 = \rank(B).
\]
If $\rank(B) = 2$, then $B = \diag(\lambda_1,\lambda_2,0)$ with 
$\lambda_1\ge\lambda_2 > 0$, and
\[
\eta_A(B) = \tfrac{1}{2}\,\diag(\lambda_2,\,\lambda_1,\,\lambda_1+\lambda_2),
\qquad \rank(\eta_A(B)) = 3 > 2 = \rank(B).
\]
Hence $\eta_A$ is indecomposable and therefore $A$ is unsplittable.
\end{exa}


\begin{proof}[Proof of part \textup{(iii)}]

$(\Rightarrow)$\;
If $A$ is LR-semisimple, there exist $L,R\in GL_n(\C)$ such that 
$L A_i R = \bigoplus_{j=1}^k B_i^{(j)}$ for all~$i$, with each 
sub-pencil $B^{(j)}$ unsplittable.  Set $c_1 = L$, $c_2 = R$; 
the scaled map $S_{c_1,c_2}(\eta_A)(X) = \sum_i (L A_i R)\,X\,(L A_i R)^*$ 
has Kraus operators $L A_i R \in \bigoplus_j M_{n_j}(\C) = \DD'$, 
where $\DD = \operatorname{span}\{P_1,\dotsc,P_k\}$ and $P_j$ 
projects onto the $j$-th block.  The $j$-th corner map is 
$\eta_{B^{(j)}}$, which is indecomposable by part~(ii) since 
$B^{(j)}$ is unsplittable.  By Lemma~\ref{lem:kraus_DCI}, 
$S_{c_1,c_2}(\eta_A)$ is $\DD$-CI.

\smallskip
$(\Leftarrow)$\;
If $S_{c_1,c_2}(\eta_A)$ is $\DD$-CI, 
Lemma~\ref{lem:kraus_DCI} gives Kraus operators 
$c_1 A_i c_2 \in \DD' = \bigoplus_j P_j M_n(\C) P_j$.  Hence 
$c_1 A_i c_2 = \bigoplus_j B_i^{(j)}$ for all~$i$, and the $j$-th 
corner map $\eta_{B^{(j)}}$ is indecomposable.  By part~(ii), 
each $B^{(j)}$ is unsplittable, so $A$ is LR-semisimple 
(with $L = c_1$, $R = c_2$).

The C-semisimple case is identical with $c_1 = c_2$.
\end{proof}

\section{Algebraic equivalences} 
\label{sec:algebraic} 

This section establishes that a Hermitian pencil is 
LR-semisimple if and only if its covariance map is 
symmetrically DS-scalable.

\subsection{Capacity theory: Indecomposable $\Rightarrow$ DS-scalable} 

\begin{lemma}[Critical points are global minimizers]
\label{lemma_critical_global_minimizers}
Let $C \succ 0$, $\det C = 1$, be a critical point of the capacity problem 
$\min\{\det(\eta(X)) : X \succ 0,\, \det X = 1\}$. Then it satisfies FOC conditions  
\begin{equation}
\label{capacity_FOC}
\eta^\ast(\eta(C)^{-1}) = C^{-1},
\end{equation}
 and 
 $\Capa(\eta) = \det(\eta(C))$, that is, $C$ is the (global) minimizer of the capacity problem.
\end{lemma}

\begin{proof}
By the method of Lagrange multipliers, at $X = C$,
\[
\nabla_X \log\det(\eta(X)) = \lambda\, \nabla_X \log\det(X).
\]
We have $\nabla_X \log\det(X) = X^{-1}$, and by 
Lemma~\ref{lemma_gradient_logdet}, 
$\nabla_X \log\det(\eta(X)) = \eta^\ast(\eta(X)^{-1})$.
Hence $\eta^\ast(\eta(C)^{-1}) = \lambda\, C^{-1}$.
Multiplying both sides by $C$ and taking traces gives
\[
\Tr\bigl(C\,\eta^\ast(\eta(C)^{-1})\bigr) = \lambda\, n.
\]
The left side equals 
$\langle \eta(C)^{-1},\, \eta(C) \rangle = \Tr(I) = n$, 
so $\lambda = 1$. The final claim follows from Lemma~\ref{lemma_exact_fixed_points_capacity} 
and $\det C = 1$.
\end{proof}

\begin{theo}[Theorem 4.7 in \cite{gurvits2004}]
\label{theo_Gurvits_4.7}
Let $\eta: M_n(\bC) \to M_n(\bC)$ be a completely positive map. 
If $\eta$ is indecomposable, then the capacity $\Capa(\eta)$ is 
strictly positive, the infimum in~\eqref{equ_defi_capacity} is 
attained, and the minimizer is unique.
\end{theo}

\begin{proof}[Proof sketch]
Positivity of capacity follows from 
Proposition~\ref{propo_RN_map_capacity}, since indecomposable 
maps are rank non-decreasing. Attainment follows from a 
compactness argument: indecomposability provides a uniform bound 
on the condition number $\gamma_1/\gamma_n$ of the eigenvalues 
of any competitor $X \succ 0$ with $\det(\eta(X)) \leq \det(\eta(I))$, 
which confines the minimizing sequence to a compact subset of 
$\{X \succ 0 : \det X = 1\}$. For uniqueness, 
Lemma~\ref{lemma_critical_global_minimizers} reduces to showing 
that the doubly stochastic case ($C = I$) has no second minimizer; 
this is established by a strict convexity argument for the function 
$(\xi_1, \ldots, \xi_n) \mapsto \log\det\bigl(\sum_i e^{\xi_i} A_i\bigr)$ 
on the hyperplane $\sum \xi_i = 0$, where $A_i = \eta(u_i u_i^*)$ 
are the images of rank-one projections in the eigenbasis of 
the putative second minimizer. See~\cite{gurvits2004} for details.
\end{proof}

\begin{coro}
\label{coro_indecomposable_DS_scalable}
If $\eta: M_n(\mathbb{C}) \to M_n(\mathbb{C})$ is indecomposable, 
then $\eta$ is DS-scalable.
\end{coro}

\begin{proof}
By Theorem~\ref{theo_Gurvits_4.7}, the capacity infimum is attained 
at some $C \succ 0$. By Lemma~\ref{lemma_critical_global_minimizers}, 
$C$ is a critical point of $X \mapsto \log\det(\eta(X)) - \log\det(X)$ 
on $\{X \succ 0 : \det X = 1\}$, so the first-order conditions give
\[
\eta^*\!\left(\eta(C)^{-1}\right) = C^{-1}.
\]
Setting $c_1 := \eta(C)^{-1/2}$ and $c_2 := C^{1/2}$, the scaled map
$\tilde\eta := S_{c_1,\, c_2}(\eta)$ satisfies
\begin{align*}
\tilde\eta(I) &= \eta(C)^{-1/2}\,\eta(C)\,\eta(C)^{-1/2} = I,
\\
\tilde\eta^*(I) &= C^{1/2}\,\eta^*\!\left(\eta(C)^{-1}\right) C^{1/2} 
= C^{1/2} C^{-1} C^{1/2} = I.
\end{align*}
Hence $\tilde\eta$ is doubly stochastic, so $\eta$ is DS-scalable 
by Definition~\ref{defi_DS_scalable}.
\end{proof}

\subsection{Uniqueness trick: self-adjoint indecomposable $\Rightarrow$ Sym DS-scalable} 


\begin{lemma}
\label{lem:indecomposable_sym_scalable}
If a completely positive map $\eta$ is self-adjoint and indecomposable then there exists a unique $C \succ 0$ with $\eta(C) = C^{-1}$.  
\end{lemma}

\begin{proof}

By Theorem \ref{theo_Gurvits_4.7}, the solution of the capacity problem is attained at some matrix $C \succ 0$ with $\det C = 1$. The FOC for this matrix is given by \eqref{capacity_FOC}. 

Since $\eta^\ast = \eta$, the FOC becomes
  \begin{equation}
 \label{selfadjoint_capacity_FOC}
 \Phi(C) := \eta\big(\eta(C)^{-1}\big)^{-1} = C.
 \end{equation}
 
 Conversely, any $C \succ 0$ with $\det C = 1$ 
satisfying~\eqref{selfadjoint_capacity_FOC} also satisfies 
the FOC~\eqref{capacity_FOC} (since $\eta^\ast = \eta$), 
and is therefore a critical point of the capacity problem. 
By Lemma~\ref{lemma_critical_global_minimizers}, it is a global 
minimizer. Since $\eta$ is indecomposable, the minimizer is unique 
by Theorem~\ref{theo_Gurvits_4.7}, so~\eqref{selfadjoint_capacity_FOC} 
has a unique solution with $\det C = 1$.

 Let $C$ be the unique solution of  \eqref{selfadjoint_capacity_FOC} that satisfies $\det(C) = 1$. Define $\tilde C := \eta(C)^{-1}$ and $\alpha := \det (\tilde C)$. Then by \eqref{selfadjoint_capacity_FOC}, $\eta\big(\eta(\tilde C)^{-1})\big)^{-1} = \tilde C$, so both $\tilde C$ and (by homogeneity of $\Phi$) $\alpha^{-1/n} \tilde C$ are solutions of \eqref{selfadjoint_capacity_FOC} and $\det(\alpha^{-1/n} \tilde C) = 1$. Set  $\beta = \alpha^{-1/n}$. By the uniqueness of the solution of \eqref{selfadjoint_capacity_FOC} we have $ C = \beta \tilde C$ which we can write as
 \[
\beta^{-1/2} C = \beta^{1/2} \eta(C)^{-1} = \eta(\beta^{-1/2} C)^{-1}. 
 \]
Hence $\beta^{-1/2} C$ is a solution of equation 
 \[
 \eta(X) X = I, 
 \]
 and this solution is unique in the cone of positive definite matrices. 

\end{proof}


\begin{coro}
\label{coro:indecomposable_sym_scalable}
If a completely positive map $\eta$ is self-adjoint and indecomposable then it is symmetrically DS-scalable.
\end{coro}
\begin{proof}
Let $C \succ 0$ be the solution of $\eta(C) = C^{-1}$. Then $S_{C^{1/2}, C^{1/2}}(\eta)$ is doubly stochastic. 
\end{proof}

\textbf{Remark.} The existence of a symmetric DS-scaling will also be established later by a different method (Proposition \ref{propo_LR_C}). What is important in Lemma \ref{lem:indecomposable_sym_scalable} is that for indecomposable maps, the scaling matrix $C$ is \emph{unique} -- a stronger conclusion than the main theorem provides, where only the uniqueness of the trace minimizer within $\mathbb{S}$ is obtained.

\subsection{LR-semisimple $\Longleftrightarrow$ DS-scalable} 

%

\begin{lemma}[Blockwise DS $\Rightarrow$ global DS for $\mathcal{D}$-bimodular maps]
\label{lemma_BDS_GDS}
Let $\eta\colon M_n(\mathbb{C})\to M_n(\mathbb{C})$ be CP and $\mathcal{D}$-bimodular,
with corner maps $\eta_j := P_j\,\eta(\,\cdot\,)\,P_j$ on $P_j M_n(\mathbb{C}) P_j$.
Suppose that for each $j$ there exist invertible
$a_j, b_j \in P_j M_n(\mathbb{C}) P_j$ such that $S_{a_j,b_j}(\eta_j)$
is doubly stochastic on $P_j M_n(\mathbb{C}) P_j$.
Setting $d_1 := \bigoplus_j a_j$ and $d_2 := \bigoplus_j b_j \in \mathcal{D}'$,
the map $\eta' := S_{d_1,d_2}(\eta)$ is doubly stochastic on $M_n(\mathbb{C})$.
\end{lemma}

\begin{proof}
Since $d_1,d_2\in\mathcal{D}'$ and $\eta$ is $\mathcal{D}$-bimodular, so is $\eta'$.
As $d_1$ (resp.\ $d_2$) restricts to $a_j$ (resp.\ $b_j$) on the $j$-th block,
the corners of $\eta'$ decouple:
\[
  (\eta')_j = S_{a_j,b_j}\!\left(\eta_j\right), \qquad j=1,\ldots,r.
\]
By assumption, each $(\eta')_j$ is doubly stochastic on $P_j M_n(\mathbb{C}) P_j$, so
\[
  (\eta')_j(P_j)=P_j
  \qquad\text{and}\qquad
  \bigl((\eta')_j\bigr)^{\!*}(P_j)=P_j.
\]
Since $P_j\in P_j M_n(\mathbb{C}) P_j$, the first identity gives $\eta'(P_j)=P_j$,
and by linearity,
\[
  \eta'(I)=\sum_{j=1}^r \eta'(P_j)=\sum_{j=1}^r P_j=I.
\]
For the adjoint: since $\mathcal{D}$ is a $*$-algebra, $\eta'^*$ is also $\mathcal{D}$-bimodular,
so $(\eta'^*)_j=\bigl((\eta')_j\bigr)^{\!*}$.
The second identity above therefore gives $\eta'^*(P_j)=P_j$, and the
same linearity argument yields $\eta'^*(I)=I$.
Hence $\eta'$ is doubly stochastic.
\end{proof}


\begin{propo}[LR-semisimple $\Rightarrow$ DS-scalable]
\label{propo_LR-semisimple_DS-scalable}
If a matrix pencil $A$ is LR-semisimple, then the corresponding CP map $\eta_A$ is DS-scalable.
\end{propo}

\begin{proof}
 By Theorem~\ref{theo_pencil_map_dictionary}, part \textup{(iii)}, $\eta_A$ is scalable to a $\DD$-CI map $\widetilde \eta_A$, with indecomposable corner blocks. These blocks are DS-scalable by Corollary \ref{coro_indecomposable_DS_scalable}. Finally, by  blockwise Lemma \ref{lemma_BDS_GDS}, we obtain that  $\widetilde \eta_A$ and therefore $\eta_A$ are DS-scalable. 
\end{proof}


 \begin{lemma}[Splitting lemma for DS maps]
\label{lemma_splitting_DS}
Let\/ $\Phi(X) = \sum_{i=1}^r B_i X B_i^*$ be a doubly stochastic 
CP map on $M_n(\bC)$, and suppose $(\RR, \LL)$ is a splitting pair 
for the pencil $\sum B_i x_i$, i.e., 
$B_i(\RR) \subseteq \LL$ for all~$i$, with 
$\dim \RR = \dim \LL = k$ and $1 \leq k \leq n{-}1$. 
Then $B_i(\RR^\perp) \subseteq \LL^\perp$ for all~$i$.
\end{lemma}

\begin{proof}
Let $P$ and $Q$ denote the orthogonal projections onto $\RR$ 
and $\LL$. The splitting condition reads $(I - Q)\,B_i\,P = 0$; 
taking adjoints gives the \emph{dual splitting}
\begin{equation}
\label{equ_dual_splitting}
P\,B_i^*\,(I - Q) = 0 \qquad \text{for all } i.
\end{equation}

Define the block operators
\begin{align*}
C_i &:= Q\,B_i\,P \colon \RR \to \LL, \qquad 
E_i := Q\,B_i\,(I{-}P) \colon \RR^\perp \to \LL, \\
D_i &:= (I{-}Q)\,B_i\,(I{-}P) \colon \RR^\perp \to \LL^\perp.
\end{align*}

By the splitting and dual splitting conditions, $B_i$ decomposes as
\begin{align}
\label{equ_Bi_decomposition}
B_i &\;=\; C_i \;+\; E_i \;+\; D_i  
\\
&\;=\; Q B_i P \;+\; Q B_i(I{-}P) \;+\; (I{-}Q) B_i(I{-}P), \notag
\end{align}
(the fourth block $(I{-}Q)B_iP$ vanishes by the splitting condition).

\medskip
\noindent\textbf{Identity 1.}\; 
Project $\sum_i B_i^* B_i = I$ by $P$ on both sides. 
Using~\eqref{equ_dual_splitting}:
\[
P\Big(\sum_i B_i^* B_i\Big)P 
= \sum_i P B_i^*\underbrace{(Q + (I{-}Q))}_I B_i P 
= \sum_i C_i^* C_i + 0 
= I_\RR,
\]
since $P B_i^*(I{-}Q) = 0$. Hence
\begin{equation}
\label{equ_identity1}
\sum_{i=1}^r C_i^* C_i = I_\RR.
\end{equation}

\medskip
\noindent\textbf{Identity 2.}\;
Project $\sum_i B_i B_i^* = I$ by $Q$ on both sides:
\[
Q\Big(\sum_i B_i B_i^*\Big)Q 
= \sum_i Q B_i\underbrace{(P + (I{-}P))}_I B_i^* Q 
= \sum_i C_i C_i^* + \sum_i E_i E_i^*
= I_\LL,
\]
using $P B_i^* Q = C_i^*$ and $(I{-}P)B_i^*Q = E_i^*$. Hence
\begin{equation}
\label{equ_identity2}
\sum_{i=1}^r C_i C_i^* + \sum_{i=1}^r E_i E_i^* = I_\LL.
\end{equation}

\medskip
\noindent\textbf{Trace argument.}\; 
Taking traces of \eqref{equ_identity1} and~\eqref{equ_identity2}, 
and using $\Tr(C_i^* C_i) = \Tr(C_i C_i^*) = \|C_i\|_F^2$:
\[
\sum_i \|C_i\|_F^2 = k, \qquad 
\sum_i \|C_i\|_F^2 + \sum_i \|E_i\|_F^2 = k.
\]
Hence $\sum_i \|E_i\|_F^2 = 0$, so $E_i = Q\,B_i\,(I{-}P) = 0$ 
for all~$i$. Together with $(I{-}Q)\,B_i\,P = 0$, this gives
\[
B_i(\RR) \subseteq \LL 
\qquad\text{and}\qquad 
B_i(\RR^\perp) \subseteq \LL^\perp 
\qquad \text{for all } i. \qedhere
\]
\end{proof}


\begin{propo}[DS pencil is LR-semisimple]
\label{propo_DS_LR_semisimple}
If\/ $\Phi(X) = \sum_{i=1}^r B_i X B_i^*$ is doubly stochastic on 
$M_n(\bC)$, then the pencil $\sum_{i=1}^r B_i\,x_i$ is 
LR-semisimple.
\end{propo}

\begin{proof}
By induction on $n$. 
The base case $n = 1$ is clear (every non-zero $1 \times 1$ pencil is 
unsplittable, hence LR-semisimple).

For the inductive step, if the pencil is unsplittable, it is 
LR-semisimple by definition. If the pencil is splittable, let 
$(\RR, \LL)$ be a splitting pair with 
$\dim \RR = \dim \LL = k$, $1 \leq k \leq n{-}1$. 
By Lemma~\ref{lemma_splitting_DS}, 
$B_i(\RR^\perp) \subseteq \LL^\perp$ for all~$i$.

Choose unitary matrices $R$ and $M$ such that $R$ maps 
$\bC^k \oplus \bC^{n-k}$ to $\RR \oplus \RR^\perp$ and 
$M$ maps $\bC^k \oplus \bC^{n-k}$ to $\LL \oplus \LL^\perp$. 
Setting $L = M^*$, the matrices $L B_i R$ are block diagonal:
\[
L B_i R = \widetilde C_i \oplus \widetilde D_i, 
\qquad 
\widetilde C_i \in M_k(\bC),\quad 
\widetilde D_i \in M_{n-k}(\bC).
\]
Since $L$ and $R$ are unitary, the sub-pencils inherit the DS 
property: from the identities established in the proof of 
Lemma~\ref{lemma_splitting_DS} (with $E_i = 0$),
\begin{alignat*}{2}
\sum_i \widetilde C_i^*\,\widetilde C_i &= I_k, &\qquad 
\sum_i \widetilde C_i\,\widetilde C_i^* &= I_k, \\[4pt]
\sum_i \widetilde D_i^*\,\widetilde D_i &= I_{n-k}, &\qquad 
\sum_i \widetilde D_i\,\widetilde D_i^* &= I_{n-k}.
\end{alignat*}
By the induction hypothesis (applied to the pencils of sizes $k$ 
and $n{-}k$, both strictly less than~$n$), both sub-pencils are 
LR-semisimple. Hence the original pencil is LR-semisimple 
(with $L$ and $R$ providing the LR-equivalence to the 
block-diagonal form).
\end{proof}


\begin{propo}[DS-scalable $\Rightarrow$ LR-semisimple]
\label{propo_DS-scalable_LR-semisimple}
If  CP map $\eta_A: X \to \sum_{i = 1}^r A_i X A_i^\ast$ is DS-scalable then  matrix pencil $A = \sum_{i = 1}^r A_i x_i$ is LR-semisimple.
\end{propo}

\begin{proof}
By hypothesis there exist $c_1, c_2 \in GL_n(\bC)$ such that 
$\Phi := S_{c_1,c_2}(\eta_A)$ is doubly stochastic. The Kraus 
operators of $\Phi$ are $B_i = c_1 A_i c_2$. By 
Proposition~\ref{propo_DS_LR_semisimple}, the pencil 
$\sum B_i\,x_i$ is LR-semisimple. Since $B_i = c_1 A_i c_2$ is 
an LR-equivalence of the original pencil, $\sum A_i\,x_i$ is 
LR-semisimple.

\end{proof}

\subsection{DS-scalable $\Rightarrow$ sym DS-scalable
  (geodesic reflection)} 

The main goal of this section is to show that for self-adjoint CP maps $\eta$, DS-scalability can be upgraded to symmetric DS-scalability. 

Let us first introduce some background material about the geometry of the cone of positively-definite matrices. 

Let $\PP_n = \{C \in M_n(\bC): C > 0\}$. We can identify the tangent space at $C \in \PP_n$ with Hermitian matrices and define the inner product
\[
g_C(H, K) := \Tr(C^{-1}H C^{-1}K). 
\] 
Then $\PP_n$ becomes a Riemannian manifold and the induced geodesic distance has the closed form
\[
d(A, B) = \Big\| \log \Big(A^{-1/2} B A^{-1/2}\Big)\Big\|_F 
= \Bigg( \sum_{j = 1}^n \Big(\log \lambda_j(A^{-1}B)\Big)^2\Big)^{1/2}, 
\]
where $\lambda_j(A^{-1}B)$ are eigenvalues. 

Properties: 
\begin{enumerate}
\item Congruence invariance: for invertible $S$, 
\[
d(S A S^\ast, S B S^\ast) = d(A, B). 
\]
\item Inversion invariance:
\[
d\big(A^{-1}, B^{-1}\big)= d(A, B).
\]
\item Geodesics are explicit:
\[
\gamma_{A \to B}(t) = A^{1/2}\Big(A^{-1/2} B A^{-1/2}\Big)^t A^{1/2}. 
\]
\end{enumerate}

In particular, for the midpoint we have an explicit formula:
\begin{align*}
A \# B & := \gamma_{A \to B}\Big(\frac{1}{2}\Big) = A^{1/2}\Big(A^{-1/2} B A^{-1/2}\Big)^{1/2} A^{1/2}\\
& =  B^{1/2}\Big(B^{-1/2} A B^{-1/2}\Big)^{1/2} B^{1/2}
\end{align*}

\medskip\noindent
\textbf{Setup.}\; Suppose $\eta$ is a self-adjoint CP map 
that is DS-scalable: there exist $C_1, C_2 \succ 0$ such 
that the scaled map $S_{C_1^{1/2},\, C_2^{1/2}}(\eta)$ is 
doubly stochastic. Unwinding the definition, the unitality 
and trace-preservation conditions read
\[
\eta(C_2) = C_1^{-1}, \qquad 
\eta^*(C_1) = C_2^{-1}.
\]
Since $\eta^* = \eta$, the second equation becomes 
$\eta(C_1) = C_2^{-1}$. Defining 
$F(X) := \eta(X)^{-1}$, these two equations say
\[
F(C_2) = C_1, \qquad F(C_1) = C_2,
\]
so $(C_1, C_2)$ is a 2-cycle of $F$ (or a fixed point if 
$C_1 = C_2$, in which case $\eta$ is already symmetrically 
DS-scalable). Our goal is to show that the geodesic midpoint 
$C_1 \mathbin{\#} C_2$ is a fixed point of $F$, which gives 
symmetric DS-scalability.

\begin{lemma}
\label{lem:geodesic}
Let $\eta(X) = \sum B_i X B_i$ with $B_i^\ast = B_i$. If $\eta(I) = P^{-1}$ and $\eta(P) = I$, then $\eta(P^t) = P^{t-1}$ for all $t \in [0,1]$.
\end{lemma}


\begin{proof}
Define
\[
\Phi(t) \;=\; P^{(1-t)/2}\,\eta(P^t)\,P^{(1-t)/2}.
\]
The hypotheses $\eta(I) = P^{-1}$ and $\eta(P) = I$ give 
$\Phi(0) = \Phi(1) = I$.  We show that $\Phi(t) = I$ for 
all $t \in [0,1]$, which is equivalent to the claim.

\medskip\noindent
\textbf{Step~1: Trace upper bound.}
Work in the eigenbasis of $P$, so that 
$P = \mathrm{diag}(\lambda_1, \ldots, \lambda_n)$ with every 
$\lambda_k > 0$.  For each $k$, define $M_k \in M_n(\mathbb{C})$ by
\[
(M_k)_{jl} \;=\; \sum_{i=1}^r (B_i)_{jk}\,\overline{(B_i)_{lk}}.
\]
Each $M_k$ is positive semidefinite (it equals 
$\sum_i b_{i,k}\,b_{i,k}^*$, where $b_{i,k}$ is the $k$-th 
column of~$B_i$).

Since $P^t = \mathrm{diag}(\lambda_1^t, \ldots, \lambda_n^t)$ 
and $B_i^* = B_i$, a direct computation gives
\[
\eta(P^t) \;=\; \sum_{k=1}^n \lambda_k^t\, M_k.
\]
As $P^{(1-t)/2}$ is diagonal with entries $\lambda_j^{(1-t)/2}$, 
the trace of $\Phi(t)$ evaluates to
\[
\mathrm{Tr}[\Phi(t)] 
\;=\; \sum_{j,k} (M_k)_{jj}\;\lambda_j^{1-t}\,\lambda_k^t.
\]
Each summand has the form 
$c_{jk}\,e^{(1-t)\log\lambda_j + t\log\lambda_k}$ 
with $c_{jk} = (M_k)_{jj} \geq 0$.  A non-negative linear 
combination of real exponentials is convex, so 
$t \mapsto \mathrm{Tr}[\Phi(t)]$ is convex on $[0,1]$.  
Since $\mathrm{Tr}[\Phi(0)] = \mathrm{Tr}[\Phi(1)] = n$, 
we conclude
\[
\mathrm{Tr}[\Phi(t)] \;\leq\; n 
\qquad\text{for all } t \in [0,1].
\]

\medskip\noindent
\textbf{Step~2: Determinant lower bound.}
The hypotheses imply that $I$ and $P$ are exact fixed points 
of the map $X \mapsto (\eta(\eta(X)^{-1}))^{-1}$.  
By Lemma~\ref{lemma_exact_fixed_points_capacity},
\[
\Capa(\eta) 
\;=\; \frac{\det(\eta(I))}{\det(I)} 
\;=\; \det(P^{-1}) 
\;=\; (\det P)^{-1}.
\]
The definition of capacity gives $\det(\eta(X)) \geq 
\Capa(\eta) \cdot \det(X)$ for all $X \succ 0$.  
Applying this to $X = P^t$:
\[
\det(\eta(P^t)) 
\;\geq\; (\det P)^{-1} \cdot (\det P)^t 
\;=\; (\det P)^{t-1},
\]
and therefore
\[
\det[\Phi(t)] 
\;=\; (\det P)^{1-t} \cdot \det(\eta(P^t)) 
\;\geq\; 1.
\]

\medskip\noindent
\textbf{Step~3: AM--GM.}
Let $\mu_1, \ldots, \mu_n > 0$ be the eigenvalues 
of~$\Phi(t)$.  By Steps~1 and~2,
\[
\frac{1}{n}\sum_{j=1}^n \mu_j 
\;=\; \frac{\mathrm{Tr}[\Phi(t)]}{n} 
\;\leq\; 1 
\;\leq\; \bigl(\det[\Phi(t)]\bigr)^{1/n} 
\;=\; \Bigl(\prod_{j=1}^n \mu_j\Bigr)^{\!1/n}.
\]
Since the arithmetic mean always dominates the geometric mean, 
both inequalities are equalities.  Equality in AM--GM forces 
$\mu_1 = \cdots = \mu_n$, and since their product is~$1$ 
while their sum is~$n$, every $\mu_j = 1$.  
Hence $\Phi(t) = I$.
\end{proof}

\begin{proposition}[Geodesic reflection for 2-cycles]
\label{prop:geodesic_reflection}
Let $\eta$ be a self-adjoint CP map, 
$\eta(X) = \sum_{i=1}^r A_i X A_i$ with $A_i^\ast = A_i$. Define $F(X) = (\eta(X))^{-1}$ 
for $X \in \mathcal{P}_n$. If $(C, D)$ is a 2-cycle of $F$, i.e., $F(C) = D$ and $F(D) = C$ 
with $C \neq D$, then $F$ reverses the geodesic $\gamma_{C \to D}$, acting as reflection 
about its midpoint $C \# D$:
\[
F(\gamma_{C \to D}(t)) = \gamma_{C \to D}(1-t) \quad \text{for all } t \in \mathbb{R}.
\]
\end{proposition}
\begin{proof}
Define $P = C^{-1/2} D C^{-1/2}$ and the conjugated CP map
\[
\widetilde{\eta}(X) = \sum_{i=1}^r B_i X B_i, \quad \text{where } B_i = C^{1/2} A_i C^{1/2}.
\]
Note that $B_i^* = B_i$ since $A_i$ and $C^{1/2}$ are both Hermitian. 
Let $\widetilde{F}(X) = (\widetilde{\eta}(X))^{-1}$.

\textbf{Step 1: Verify that $(I, P)$ is a 2-cycle of $\widetilde{F}$.}
We have
\begin{align*}
\widetilde{\eta}(I) &= \sum_i B_i^2 
= \sum_i (C^{1/2} A_i C^{1/2})^2 
= \sum_i C^{1/2} A_i C A_i C^{1/2} 
\\
&= C^{1/2} \eta(C) C^{1/2} = C^{1/2} D^{-1} C^{1/2},
\end{align*}
using $F(C) = D$, i.e., $\eta(C) = D^{-1}$. Therefore
\[
\widetilde{F}(I) = (C^{1/2} D^{-1} C^{1/2})^{-1} = C^{-1/2} D C^{-1/2} = P.
\]
Similarly,
\begin{align*}
\widetilde{\eta}(P) &= \sum_i B_i P B_i 
= \sum_i C^{1/2} A_i C^{1/2} \cdot C^{-1/2} D C^{-1/2} \cdot C^{1/2} A_i C^{1/2} 
\\
&= C^{1/2} \eta(D) C^{1/2} 
= C^{1/2} C^{-1} C^{1/2} = I,
\end{align*}
using $F(D) = C$, i.e., $\eta(D) = C^{-1}$. Therefore $\widetilde{F}(P) = I$.

\textbf{Step 2: Apply Lemma~\ref{lem:geodesic}.}
Since $\widetilde{\eta}$ is self-adjoint and satisfies 
$\widetilde{F}(I) = P$, $\widetilde{F}(P) = I$, Lemma~\ref{lem:geodesic} gives
\[
\widetilde{F}(P^t) = P^{1-t} \quad \text{for all } t \in \mathbb{R}.
\]

\textbf{Step 3: Transfer back to original coordinates.}
The relation between $\widetilde{\eta}$ and $\eta$ is
\[
\widetilde{\eta}(X) = C^{1/2} \eta(C^{1/2} X C^{1/2}) C^{1/2},
\]
which gives $F(C^{1/2} X C^{1/2}) = C^{1/2} \widetilde{F}(X) C^{1/2}$.
Setting $X = P^t$ and using $\gamma_{C \to D}(t) = C^{1/2} P^t C^{1/2}$:
\[
F(\gamma_{C \to D}(t)) = C^{1/2} \widetilde{F}(P^t) C^{1/2} 
= C^{1/2} P^{1-t} C^{1/2} = \gamma_{C \to D}(1-t). \qedhere
\] 
\end{proof}
\begin{remark}
The midpoint of the geodesic is $\gamma_{C \to D}(1/2) = C \# D$, the geometric mean of $C$ and $D$. 
Setting $t = 1/2$ shows that $C \# D$ is a fixed point of $F$:
\[
F(C \# D) = C \# D.
\]
Thus $F$ fixes the midpoint and swaps $C = \gamma_{C \to D}(0)$ and $D = \gamma_{C \to D}(1)$, 
consistent with the 2-cycle condition.  When $C = D$, i.e., $C$ is itself a fixed point, the 
identity holds trivially with $P = I$.
\end{remark}

\begin{propo}
\label{propo_LR_C}
Let $A = \sum_{i=1}^r A_i x_i$, $A_i^* = A_i$, be an LR-semisimple Hermitian matrix pencil.  
Then its covariance map $\eta_A$ is symmetrically DS-scalable.
\end{propo}
\begin{proof}
By Proposition~\ref{propo_LR-semisimple_DS-scalable}, LR-semisimplicity of $A$ implies that 
$\eta_A$ is DS-scalable: there exist $C_1, C_2 \in \mathcal{P}_n$ such that
\[
\eta_A(C_2) = C_1^{-1}, \qquad \eta_A^*(C_1) = C_2^{-1}.
\]
Since the $A_i$ are Hermitian, $\eta_A$ is self-adjoint, so the second condition becomes 
$\eta_A(C_1) = C_2^{-1}$.  Writing $F(X) = (\eta_A(X))^{-1}$, this gives
\[
F(C_1) = C_2, \qquad F(C_2) = C_1,
\]
so $(C_1, C_2)$ is a 2-cycle of $F$ (or a fixed point if $C_1 = C_2$).  
By Proposition~\ref{prop:geodesic_reflection}, the geometric mean $C_1 \# C_2$ is a fixed 
point of $F$, i.e., $\eta_A(C_1 \# C_2) = (C_1 \# C_2)^{-1}$.  Setting $C = C_1 \# C_2$, 
this is precisely the condition that $\eta_A$ is symmetrically DS-scalable.
\end{proof}

\section{The analytic engine} 
\label{sec:analytic_engine}

This section establishes the forward analytic implication: 
symmetric DS-scalability implies non-singularity at $t = 0$.

\subsection{The solution manifold} 


\begin{propo}[The $(-1)$-eigenspace generates a solution family]
\label{propo:kernel_tangent_family}
Let $\eta\colon M_n(\mathbb{C}) \to M_n(\mathbb{C})$ be a 
self-adjoint CP map with $\eta(C) = C^{-1}$ for some $C \succ 0$, 
and let $\Phi = S_{C^{1/2},C^{1/2}}(\eta)$ be the associated 
doubly stochastic map. Let 
$\mathcal{V} := \ker(\Phi + \mathrm{Id})|_{\mathrm{Herm}_n}$ 
and $d := \dim \mathcal{V}$.
\begin{enumerate}[label=\textup{(\roman*)}]
\item \label{item:family}
The map $\Phi$ restricts to a $*$-homo\-morphism on the 
$C^*$-algebra $C^*(\mathcal{V})$, and the map
\[
\mathcal{V} \ni Y \;\longmapsto\; V_Y := C^{1/2}\,e^{Y}\,C^{1/2}
\]
parametrizes a $d$-dimensional manifold $\mathcal{M}$ of positive 
definite solutions of $\eta(V)\,V = I$, with $V_0 = C$.

\item \label{item:tangent}
The tangent space of $\mathcal{M}$ at $C$ is exactly the kernel 
of the linearization $\mathcal{L}\colon K \mapsto K + C\,\eta(K)\,C$:
\[
T_C\mathcal{M} 
= \bigl\{\,C^{1/2}Y\,C^{1/2} : Y \in \mathcal{V}\,\bigr\} 
= \ker\mathcal{L}.
\]
\end{enumerate}
\end{propo}

\begin{proof}
\ref{item:family}\;
Let $Y_1, \ldots, Y_d$ be a basis of $\mathcal{V}$. By 
Lemma~\ref{lemma_unit_eigenstates}, each $Y_j$ satisfies 
$\Phi(Y_j^2) = \Phi(Y_j)^2$. By Theorem~\ref{theo_choi}, 
this gives $\Phi(X Y_j) = \Phi(X)\,\Phi(Y_j)$ and 
$\Phi(Y_j X) = \Phi(Y_j)\,\Phi(X)$ for all 
$X \in M_n(\mathbb{C})$ and every~$j$. In particular, $\Phi$ 
is multiplicative on all products of elements from 
$\{Y_1, \ldots, Y_d\}$, so $\Phi$ restricts to a 
$*$-homo\-morphism on $C^*(\mathcal{V})$.

 Since for any $Y \in \mathcal{V}$, the exponential $e^Y$ lies in 
$C^*(\mathcal{V})$, so
\[
\Phi(e^Y) = e^{\Phi(Y)} = e^{-Y} = (e^Y)^{-1}.
\]
Set \(X=e^Y\) and \(V=C^{1/2}XC^{1/2}\). Then
\[
\Phi(X)=C^{1/2}\eta(C^{1/2}XC^{1/2})C^{1/2}
      =C^{1/2}\eta(V)C^{1/2},
\]
so
\[
\eta(V)=C^{-1/2}\Phi(X)C^{-1/2}
       =C^{-1/2}X^{-1}C^{-1/2}
       =V^{-1}.
\]
Hence \(\eta(V)\,V=I\).

Since \(\mathcal V\subset \Herm_n\) is a real vector space of dimension \(d\), the map
\[
\mathcal V\ni Y \longmapsto C^{1/2}e^Y C^{1/2}
\]
is a real-analytic embedding (the exponential map is injective on Hermitian matrices and maps \(\Herm_n\) into \(M_n(\mathbb C)_{>0}\)); therefore its image \(\mathcal M\) is a \(d\)-dimensional real-analytic embedded submanifold of \(M_n(\mathbb C)_{>0}\).

\ref{item:tangent}\;
Differentiating $V_{tY} = C^{1/2}e^{tY}C^{1/2}$ at $t = 0$ 
gives $\frac{d}{dt}\big|_{t=0} V_{tY} = C^{1/2}YC^{1/2}$. 
This spans $\ker\mathcal{L}$ as $Y$ ranges over $\mathcal{V}$, 
since $\mathcal{L}(C^{1/2}YC^{1/2}) = C^{1/2}(Y + \Phi(Y))C^{1/2} = 0$ 
if and only if $Y \in \mathcal{V}$.
\end{proof}

\subsection{Lyapunov--Schmidt reduction} 


\begin{lemma}[Cokernel of the linearization operator]
\label{lem:cokernel}
Let $\eta\colon M_n(\mathbb{C}) \to M_n(\mathbb{C})$ be a 
self-adjoint CP map with $\eta(C) = C^{-1}$ for some $C \succ 0$, 
and let $\Phi = S_{C^{1/2},C^{1/2}}(\eta)$. Then the adjoint of 
the linearization operator 
$\mathcal{L}(K) = K + C\,\eta(K)\,C$ with respect to the real inner product 
$\langle A, B \rangle = \operatorname{Tr}(AB)$ on $\Herm_n$ is
\[
\mathcal{L}^*(L) = L + \eta(CLC),
\]
and
\[
\ker\mathcal{L}^* 
= \bigl\{\, C^{-1/2}Y\,C^{-1/2} : 
  Y \in \ker(\Phi + \mathrm{Id})|_{\mathrm{Herm}_n}\,\bigr\}.
\]
\end{lemma}

\begin{proof}
For Hermitian $K$ and $L$:
\[
\langle L,\, \mathcal{L}(K) \rangle 
= \operatorname{Tr}(LK) 
  + \operatorname{Tr}(LC\,\eta(K)\,C)
= \operatorname{Tr}(LK) 
  + \operatorname{Tr}(CLC \cdot \eta(K)),
\]
where the second equality uses cyclicity of the trace. Since 
$\eta$ is self-adjoint ($\eta^* = \eta$, which holds because the 
Kraus operators are Hermitian), 
$\operatorname{Tr}(CLC \cdot \eta(K)) 
= \operatorname{Tr}(\eta(CLC) \cdot K)$. Therefore
\[
\langle L,\, \mathcal{L}(K) \rangle 
= \operatorname{Tr}\!\big((L + \eta(CLC))\,K\big) 
= \langle \mathcal{L}^*(L),\, K \rangle,
\]
giving $\mathcal{L}^*(L) = L + \eta(CLC)$.

To find $\ker\mathcal{L}^*$, substitute $L = C^{-1/2}ZC^{-1/2}$ 
with $Z$ Hermitian:
\begin{align*}
\mathcal{L}^*(L) 
&= C^{-1/2}ZC^{-1/2} 
  + \eta\!\big(C \cdot C^{-1/2}ZC^{-1/2} \cdot C\big) \\
&= C^{-1/2}ZC^{-1/2} 
  + \eta\!\big(C^{1/2}ZC^{1/2}\big).
\end{align*}
Using $\eta(C^{1/2}ZC^{1/2}) 
= C^{-1/2}\Phi(Z)\,C^{-1/2}$ (from the definition of $\Phi$):
\[
\mathcal{L}^*(L) 
= C^{-1/2}\!\big(Z + \Phi(Z)\big)\,C^{-1/2}.
\]
Since $C^{-1/2}$ is invertible, $\mathcal{L}^*(L) = 0$ if and 
only if $\Phi(Z) = -Z$.
\end{proof}


\medskip\noindent
\textbf{Lyapunov--Schmidt decomposition.}\;
Consider the equation 
\[
\Psi(W, u) := W - (\eta(W) + uI)^{-1} = 0
\]
on $\operatorname{Herm}_n \times \mathbb{R}$, with 
$\Psi(C, 0) = 0$. The Fr\'echet derivative at $(C, 0)$ is 
$D_W\Psi|_{(C,0)} = \LL$, the linearization operator 
$\LL(K) = K + C\,\eta(K)\,C$. 

Assume that 
$\dim\ker\LL = d \geq 1$. Let $\{Y_1, \ldots, Y_d\}$ be a 
basis of $\mathcal{V} := \ker(\Phi + \Id)|_{\Herm_n}$. By 
Proposition~\ref{propo:kernel_tangent_family}, 
$\ker\LL = \operatorname{span}\{K_{0,1}, \ldots, K_{0,d}\}$ 
with $K_{0,k} = C^{1/2}Y_k C^{1/2}$. By 
Lemma~\ref{lem:cokernel}, 
$\ker\LL^* = \operatorname{span}\{L_{0,1}, \ldots, L_{0,d}\}$ 
with $L_{0,j} = C^{-1/2}Y_j C^{-1/2}$.

Write $W = C + \sum_k s_k K_{0,k} + \tilde{K}$ with 
$\tilde{K} \perp \operatorname{span}\{K_{0,k}\}$ in 
$\operatorname{Herm}_n$. Since $\LL$ is invertible on 
$(\ker\LL)^{\perp}$, the implicit function theorem solves the 
\emph{auxiliary equation} (the projection of $\Psi = 0$ onto 
$(\ker\LL^*)^{\perp}$) for 
$\tilde{K} = \tilde{K}(\mathbf{s}, u)$, smooth near the origin. 

The \emph{bifurcation equations} are the remaining $d$ scalar 
equations
\[
\phi_j(\mathbf{s}, u) 
:= \left\langle L_{0,j},\, 
   \Psi\Big(C + {\textstyle\sum_k} s_k K_{0,k} 
   + \tilde{K}(\mathbf{s},u),\, u\Big) \right\rangle = 0, 
\qquad j = 1, \ldots, d.
\]

\medskip\noindent
\textbf{Hadamard factorization.}\;
Each bifurcation equation $\phi_j(\mathbf{s}, u) = 0$ vanishes 
identically at $u = 0$ (since the family 
$V_\mathbf{t} = C^{1/2}\exp(\sum_k t_k Y_k)\,C^{1/2}$ provides 
solutions of $\Psi = 0$ at $u = 0$ that sweep out a neighborhood 
in the $\ker\LL$-directions). This gives the factorization 
$\phi_j(\mathbf{s}, u) = u\,\psi_j(\mathbf{s}, u)$ with smooth 
$\psi_j$, for each $j = 1, \ldots, d$. (This is the Hadamard 
lemma applied to each $\phi_j$: if $f(x, 0) = 0$ for all $x$ 
near $0$ and $f$ is smooth, then $f(x, y) = y \cdot g(x, y)$ 
for a smooth $g$, given explicitly by 
$g(x, y) = \int_0^1 \partial_y f(x, ty)\,dt$.)

\subsection{Trace minimizer and transversality} 
\begin{propo}[Trace minimizer in $\bS$]
\label{propo:trace_minimizer}
Let $\eta\colon M_n(\bC) \to M_n(\bC)$ be a self-adjoint CP 
map with $\bS := \{W \succ 0 : \eta(W)\,W = I\} \neq 
\emptyset$. Then:
\begin{enumerate}[label=\textup{(\roman*)}]
\item \textup{(Existence)} $\bS$ has an element of minimal 
  trace: there exists $C \in \bS$ with 
  $\tr(C) \leq \tr(W)$ for all $W \in \bS$.
\item \textup{(First-order conditions)} Let 
  $\Phi = S_{C^{1/2},C^{1/2}}(\eta)$ and 
  $\VV = \ker(\Phi + \Id)|_{\Herm_n}$ with basis 
  $\{Y_1, \ldots, Y_d\}$. Then
  \[
  \Tr(Y_j\,C) = 0 \qquad \text{for all } 
  j = 1, \ldots, d.
  \]
\end{enumerate}
\end{propo}

\begin{proof}
\textbf{(i).}\; Let $\{W_k\} \subset \bS$ be a minimizing 
sequence with 
$\tr(W_k) \to \inf_{W \in \bS}\tr(W)$. Since 
$\lambda_{\max}(W_k) \leq \Tr(W_k) \leq M$ and 
\[
\lambda_{\min}(W_k) \geq 1/(M\,\|\eta(I)\|) > 0
\]
(using $\|W_k^{-1}\| = \|\eta(W_k)\| \leq M\|\eta(I)\|$), 
a subsequence converges to some $C$ with $C \succ 0$ and 
$\eta(C)\,C = I$.

\medskip\noindent
\textbf{(ii).}\; By 
Proposition~\ref{propo:kernel_tangent_family}, 
$V_s := C^{1/2}\,e^{sY_j}\,C^{1/2} \in \bS$ for all 
$s \in \bR$. Since $C$ minimizes trace over~$\bS$:
\[
\frac{d}{ds}\bigg|_{s=0} \Tr(C\,e^{sY_j}) 
= \Tr(C\,Y_j) = 0. \qedhere
\]
\end{proof}


\begin{lemma}[ Solvability of bifurcation equations]
\label{lem:transversality_general}
$\,$ \\
Let $\eta\colon M_n(\bC) \to M_n(\bC)$ be a self-adjoint CP map 
with $\eta(C) = C^{-1}$ for some $C \succ 0$, and let 
$\Phi = S_{C^{1/2},C^{1/2}}(\eta)$. Let 
$\mathcal{V} = \ker(\Phi + \Id)|_{\Herm_n}$ with 
$d = \dim\mathcal{V} \geq 1$, and let $\{Y_1, \ldots, Y_d\}$ 
be an orthogonal basis of $\mathcal{V}$. Suppose that 
\begin{equation}\label{eq:FOC_general}
\Tr(Y_j C) = 0 \qquad \text{for all } j = 1, \ldots, d.
\end{equation}
Let $\phi_j(\mathbf{s}, u) = u\,\psi_j(\mathbf{s}, u)$ be the 
Hadamard factorization of the Lyapunov--Schmidt bifurcation 
equations at $(C, 0)$. Then:
\begin{enumerate}[label=\textup{(\roman*)}]
\item $\psi_j(0, 0) = 0$ for all $j$.
\item The Jacobian 
\[
\dfrac{\partial\psi_j}{\partial s_k}\bigg|_{(0,0)} 
= \frac{1}{2}\big[\Tr(Y_j C Y_k) + \Tr(Y_k C Y_j) \big],
\]
 and this matrix is positive definite.
\end{enumerate}
In particular, the implicit function theorem applies to the 
system $\psi_j(\mathbf{s}, u) = 0$ at $(\mathbf{s}, u) = (0, 0)$, 
giving a smooth curve $\mathbf{s}(u)$ with $\mathbf{s}(0) = 0$.
\end{lemma}

\begin{proof}
We use the notation
\[
K_{0,k} := C^{1/2}Y_k C^{1/2} \in \ker\LL, \qquad 
L_{0,j} := C^{-1/2}Y_j C^{-1/2} \in \ker\LL^*
\]
(Lemma~\ref{lem:cokernel}). The Lyapunov--Schmidt decomposition 
writes $W = C + \sum_k s_k K_{0,k} + \tilde{K}$ with 
$\tilde{K} \perp \operatorname{span}\{K_{0,k}\}$, and the 
bifurcation equations are
\[
\phi_j(\mathbf{s}, u) = \Tr\!\big(L_{0,j} \cdot 
\Psi(C + {\textstyle\sum_k} s_k K_{0,k} 
+ \tilde{K}(\mathbf{s}, u),\, u)\big).
\]
We record three identities used throughout. Define 
$\hat{Q}_{jk} := C^{-1/2}Y_j Y_k C^{-1/2}$.


\medskip\noindent
\textbf{Identity 1:} $C\,\eta(K_{0,k})\,C = -K_{0,k}$.

\noindent\textit{Proof.}\; 
$\eta(K_{0,k}) = C^{-1/2}\Phi(Y_k)C^{-1/2} = -L_{0,k}$, so 
$C\,\eta(K_{0,k})\,C = -C\,L_{0,k}\,C = -K_{0,k}$.

\medskip\noindent
\textbf{Identity 2:} $\Phi(Y_j Y_k) = Y_j Y_k$ for all $j, k$.

\noindent\textit{Proof.}\; Since $\Phi(Y_\ell) = -Y_\ell$ with 
$|-1| = 1$ for each $\ell$, the multiplicative domain theorem 
gives that $\Phi$ restricts to a $*$-homomorphism on 
$C^*(Y_1, \ldots, Y_d)$. Therefore 
$\Phi(Y_j Y_k) = \Phi(Y_j)\Phi(Y_k) = (-Y_j)(-Y_k) = Y_j Y_k$.

\medskip\noindent
\textbf{Identity 3:} $\eta(C\hat{Q}_{jk}C) = \hat{Q}_{jk}$.

\noindent\textit{Proof.}\; 
$C\hat{Q}_{jk}C = C^{1/2}Y_j Y_k C^{1/2}$, so 
$\eta(C\hat{Q}_{jk}C) = C^{-1/2}\Phi(Y_j Y_k)C^{-1/2} 
= C^{-1/2}Y_j Y_k C^{-1/2} = \hat{Q}_{jk}$ by Identity~2.


\bigskip\noindent
\textbf{Part (i): $\psi_j(0,0) = 0$.}\;
By definition, $\psi_j(0,0) = (\phi_j)_u(0,0)$. We have
\[
\phi_j(\mathbf{s},u) = \Tr\!\big(L_{0,j} \cdot 
\Psi(C + {\textstyle\sum_k} s_k K_{0,k} 
+ \tilde{K}(\mathbf{s},u),\, u)\big).
\]
Differentiating in $u$ at $(\mathbf{s},u) = (\mathbf{s},0)$ 
by the chain rule:
\[
(\phi_j)_u(\mathbf{s},0) = 
\Tr\!\big(L_{0,j} \cdot 
D_W\Psi(W_\mathbf{s},0)[\tilde{K}_u(\mathbf{s},0)]\big) 
+ \Tr\!\big(L_{0,j} \cdot D_u\Psi(W_\mathbf{s},0)\big),
\]
where 
$W_\mathbf{s} := C + \sum_k s_k K_{0,k} 
+ \tilde{K}(\mathbf{s},0)$. At $\mathbf{s} = 0$: 
$W_0 = C$ and $D_W\Psi(C,0) = \LL$, so the first term becomes 
$\Tr(L_{0,j} \cdot \LL[\tilde{K}_u(0,0)]) 
= \Tr(\LL^*(L_{0,j}) \cdot \tilde{K}_u(0,0)) = 0$ since 
$L_{0,j} \in \ker\LL^*$. The second term is 
$\Tr(L_{0,j} \cdot C^2) = \Tr(Y_j C) = 0$ 
by~\eqref{eq:FOC_general}. Therefore 
$\psi_j(0,0) = (\phi_j)_u(0,0) = 0$.


\bigskip\noindent
\textbf{Part (ii): the Jacobian.}\;
We compute $\partial\psi_j/\partial s_k|_{(0,0)} 
= (\phi_j)_{s_k u}(0,0)$. This involves 
several steps.

\medskip\noindent
\textbf{Step 1: $\tilde{K}_{s_k}(0,0) = 0$.}\;
The family $V_\mathbf{t} = C^{1/2}\exp(\sum_\ell t_\ell Y_\ell)\,
C^{1/2}$ solves $\Psi(V_\mathbf{t}, 0) = 0$ for all 
$\mathbf{t}$. Consider the one-parameter subfamily 
$V_k(t) := V_{t\mathbf{e}_k}$ for fixed $k$. Expanding at 
$t = 0$:
\[
V_k(t) = C + t\,K_{0,k} + O(t^2).
\]
The Lyapunov--Schmidt coordinate of $V_k(t)$ along $K_{0,k}$ is
\[
\sigma_k(t) = \frac{\langle K_{0,k},\, V_k(t) - C\rangle}
{\|K_{0,k}\|^2} = t + O(t^2),
\]
so $\sigma_k(0) = 0$ and $\sigma_k'(0) = 1$. The 
$(\ker\LL)^\perp$-component of $V_k(t) - C$ is
\[
\tilde{K}_{0,k}(t) := (V_k(t) - C) - \sigma_k(t)\,K_{0,k},
\]
which satisfies $\tilde{K}_{0,k}(0) = 0$ and 
$\tilde{K}_{0,k}'(0) = K_{0,k} - \sigma_k'(0)\,K_{0,k} = 0$. 
Since $V_k(t)$ solves $\Psi = 0$ at $u = 0$, the auxiliary 
equation is satisfied along the subfamily, giving 
$\tilde{K}(\sigma_k(t)\,\mathbf{e}_k,\, 0) 
= \tilde{K}_{0,k}(t)$. Differentiating at $t = 0$:
\[
\tilde{K}_{s_k}(0,0) \cdot \sigma_k'(0) 
= \tilde{K}_{0,k}'(0) = 0.
\]
Since $\sigma_k'(0) = 1$, we conclude 
$\tilde{K}_{s_k}(0,0) = 0$.

(The family $V_\mathbf{t}$ provides exact solutions at $u = 0$ whose tangent vectors at the base point $C$ are spanned by  the kernel directions $K_{0,k}$. Moving along the family in the $k$-th direction moves purely along $K_{0,k}$ to first order, with no component orthogonal to the kernel. So the LS ``remainder" $\tilde{K}$ doesn't activate at first order in $s_k$.)

\medskip\noindent
\textbf{Step 2: decomposition of 
$(\phi_j)_{s_k u}(0,0)$.}\;
By definition,
\[
\phi_j(\mathbf{s},u) = \Tr\!\big(L_{0,j} \cdot 
\Psi(W(\mathbf{s},u),\, u)\big),
\]
where $W(\mathbf{s},u) = C + \sum_\ell s_\ell K_{0,\ell} 
+ \tilde{K}(\mathbf{s},u)$. Differentiating in $s_k$:
\[
(\phi_j)_{s_k} = \Tr\!\big(L_{0,j} \cdot 
D_W\Psi(W,u)[W_{s_k}]\big),
\]
where $W_{s_k} = K_{0,k} + \tilde{K}_{s_k}$. Differentiating 
again in $u$ at $(0,0)$, the product rule gives three terms:
\begin{align*}
(\phi_j)_{s_k u}(0,0) 
&= \Tr\!\big(L_{0,j} \cdot 
   D^2_{WW}\Psi(C,0)[\tilde{K}_u(0,0),\, K_{0,k}]\big) \\
&\quad + \Tr\!\big(L_{0,j} \cdot 
   D^2_{Wu}\Psi(C,0)[K_{0,k}]\big) \\
&\quad + \Tr\!\big(L_{0,j} \cdot 
   \LL[\tilde{K}_{s_k u}(0,0)]\big),
\end{align*}
where we used $W_{s_k}(0,0) = K_{0,k}$ (Step~1), 
$W_u(0,0) = \tilde{K}_u(0,0)$, and 
$D_W\Psi(C,0) = \LL$. The third term vanishes since 
$L_{0,j} \in \ker\LL^*$, leaving
\begin{equation}\label{eq:jac_decomp}
\frac{\partial\psi_j}{\partial s_k}\bigg|_{(0,0)} 
= \underbrace{\Tr\!\big(L_{0,j} \cdot 
  D^2_{WW}\Psi(C,0)[K_{0,k},\, \tilde{K}_u]\big)}_
  {\text{(I)}_{jk}} 
+ \underbrace{\Tr\!\big(L_{0,j} \cdot 
  D^2_{Wu}\Psi(C,0)[K_{0,k}]\big)}_{\text{(II)}_{jk}},
\end{equation}
where $\tilde{K}_u = \tilde{K}_u(0,0)$.

\medskip\noindent
\textbf{Step 3: the auxiliary equation for $\tilde{K}_u$.}\;
The auxiliary equation is 
$P[\Psi(W(\mathbf{s},u), u)] = 0$, where $P$ projects onto 
$(\ker\LL^*)^\perp$. Differentiating in $u$ at $(0,0)$:
\[
P\big[\LL(\tilde{K}_u) + C^2\big] = 0,
\]
i.e., $\LL(\tilde{K}_u) + C^2 \in \ker P 
= \operatorname{span}\{L_{0,1}, \ldots, L_{0,d}\}$. But 
$\langle L_{0,j}, \LL(\tilde{K}_u)\rangle = 0$ (since 
$L_{0,j} \in \ker\LL^*$) and 
$\langle L_{0,j}, C^2\rangle = \Tr(Y_j C) = 0$ (assumption). 
So the component of $\LL(\tilde{K}_u) + C^2$ in each 
$L_{0,j}$-direction vanishes, giving
\begin{equation}\label{eq:aux_general}
\LL(\tilde{K}_u) = -C^2,
\end{equation}
equivalently $\tilde{K}_u + C\,\eta(\tilde{K}_u)\,C = -C^2$.

\medskip\noindent
\textbf{Step 4: a key trace identity.}\;
Taking the inner product of~\eqref{eq:aux_general} with 
$\hat{Q}_{jk}$:
\[
\Tr(\hat{Q}_{jk}\,\tilde{K}_u) 
+ \Tr(\hat{Q}_{jk}\,C\,\eta(\tilde{K}_u)\,C) 
= -\Tr(\hat{Q}_{jk}\,C^2).
\]
For the second term, cyclicity and self-adjointness of $\eta$ 
give
\[
\Tr(\hat{Q}_{jk}\,C\,\eta(\tilde{K}_u)\,C) 
= \Tr(\eta(C\hat{Q}_{jk}C)\,\tilde{K}_u) 
= \Tr(\hat{Q}_{jk}\,\tilde{K}_u),
\]
where the last step uses Identity~3. Therefore
\begin{equation}\label{eq:key_trace_general}
\Tr(\hat{Q}_{jk}\,\tilde{K}_u) 
= -\tfrac{1}{2}\,\Tr(\hat{Q}_{jk}\,C^2) 
= -\tfrac{1}{2}\,\Tr(Y_j Y_k C).
\end{equation}

\medskip\noindent
\textbf{Step 5: computing term $\text{(II)}_{jk}$.}\;
We need $D^2_{Wu}\Psi(C,0)[K]$, the derivative of 
$D_u\Psi(W,u) = (\eta(W)+uI)^{-2}$ with respect to $W$ in 
direction $K$, evaluated at $(C,0)$. Set 
$A(t) := \eta(C + tK) + 0 \cdot I = C^{-1} + t\,\eta(K)$. 
Using the product rule on $A^{-2} = A^{-1} \cdot A^{-1}$:
\[
\frac{d}{dt}\Big|_{t=0} A(t)^{-2} 
= \frac{d}{dt}\Big|_{t=0}\!\big(A^{-1}\big) \cdot A(0)^{-1} 
+ A(0)^{-1} \cdot \frac{d}{dt}\Big|_{t=0}\!\big(A^{-1}\big).
\]
Since $\frac{d}{dt}|_{t=0}\, A(t)^{-1} 
= -A(0)^{-1}\,\dot{A}(0)\,A(0)^{-1} 
= -C\,\eta(K)\,C$ (using $A(0) = C^{-1}$, $A(0)^{-1} = C$, 
$\dot{A}(0) = \eta(K)$), we obtain
\begin{align*}
D^2_{Wu}\Psi(C,0)[K] 
&= -C\,\eta(K)\,C \cdot C + C \cdot \big({-C\,\eta(K)\,C}\big)
\\
&= -C\,\eta(K)\,C^2 - C^2\,\eta(K)\,C.
\end{align*}

For $K = K_{0,k}$, using $\eta(K_{0,k}) = -L_{0,k}$ 
(Identity~1):
\[
D^2_{Wu}\Psi(C,0)[K_{0,k}] = C\,L_{0,k}\,C^2 + C^2\,L_{0,k}\,C.
\]
Therefore
\begin{align*}
\text{(II)}_{jk} 
&= \Tr(L_{0,j}\,C\,L_{0,k}\,C^2) 
  + \Tr(L_{0,j}\,C^2\,L_{0,k}\,C).
\end{align*}
For the first summand: 
$L_{0,j}\,C\,L_{0,k} = C^{-1/2}Y_j Y_k C^{-1/2} 
= \hat{Q}_{jk}$, so 
\[
\Tr(L_{0,j}\,C\,L_{0,k}\,C^2) = \Tr(\hat{Q}_{jk}\,C^2) 
= \Tr(Y_j Y_k C).
\] 
For the second: 
$L_{0,j}\,C^2\,L_{0,k}\,C 
= C^{-1/2}Y_j\,C\,Y_k\,C^{1/2}$ 
(using $L_{0,j}\,C = C^{-1/2}Y_j C^{1/2}$ and 
$C\,L_{0,k}\,C = K_{0,k}$), so 
$\Tr(L_{0,j}\,C^2\,L_{0,k}\,C) = \Tr(Y_j\,C\,Y_k)$. Thus
\begin{equation}\label{eq:II_general}
\text{(II)}_{jk} = \Tr(Y_k C Y_j ) + \Tr( Y_j C Y_k ).
\end{equation}

\medskip\noindent
\textbf{Step 6: computing term $\text{(I)}_{jk}$.}\;
Since $\Psi(W,0) = W - (\eta(W))^{-1}$, differentiating 
$D_W\Psi(W,0)[K_2] = K_2 + (\eta(W))^{-1}\eta(K_2)(\eta(W))^{-1}$ 
in $W$ at $C$ by the product rule on $A^{-1}\eta(K_2)A^{-1}$ 
with $A = \eta(W)$:
\[
D^2_{WW}\Psi(C,0)[K_1, K_2] 
= -C\,\eta(K_1)\,C\,\eta(K_2)\,C 
  - C\,\eta(K_2)\,C\,\eta(K_1)\,C.
\]
For $K_1 = K_{0,k}$, using $C\,\eta(K_{0,k})\,C = -K_{0,k}$ 
(Identity~1):
\[
D^2_{WW}\Psi(C,0)[K_{0,k},\, \tilde{K}_u] 
= K_{0,k}\,\eta(\tilde{K}_u)\,C 
+ C\,\eta(\tilde{K}_u)\,K_{0,k}.
\]
From~\eqref{eq:aux_general}: 
$C\,\eta(\tilde{K}_u)\,C = -C^2 - \tilde{K}_u$, so 
$\eta(\tilde{K}_u)\,C = -C - C^{-1}\tilde{K}_u$ and 
$C\,\eta(\tilde{K}_u) = -C - \tilde{K}_u C^{-1}$. Substituting:
\begin{align*}
\text{(I)}_{jk} 
&= \Tr\!\big(L_{0,j}\,K_{0,k}
   (-C - C^{-1}\tilde{K}_u)\big) 
+ \Tr\!\big(L_{0,j}
   (-C - \tilde{K}_u C^{-1})\,K_{0,k}\big) \\
&= -\Tr(L_{0,j} K_{0,k} C) 
   - \Tr(L_{0,j} K_{0,k} C^{-1}\tilde{K}_u) \\
&\quad\; -\Tr(L_{0,j} C K_{0,k}) 
   - \Tr(L_{0,j} \tilde{K}_u C^{-1} K_{0,k}).
\end{align*}

For the terms without $\tilde{K}_u$: 
$L_{0,j}\,K_{0,k} = C^{-1/2}Y_j Y_k C^{1/2}$, so 
$\Tr(L_{0,j} K_{0,k} C) = \Tr(Y_j Y_k C)$, and 
$\Tr(L_{0,j} C K_{0,k}) = \Tr(Y_j C Y_k)$ (since 
$L_{0,j}\,C = C^{-1/2}Y_j C^{1/2}$ and 
$C\,K_{0,k} = C^{3/2}Y_k C^{1/2}$). 

For the terms with 
$\tilde{K}_u$: 
$L_{0,j}\,K_{0,k}\,C^{-1} 
= C^{-1/2}Y_j Y_k C^{-1/2} = \hat{Q}_{jk}$, and by cyclicity 
$\Tr(L_{0,j}\,\tilde{K}_u\,C^{-1}\,K_{0,k}) 
= \Tr(C^{-1}K_{0,k}\,L_{0,j}\,\tilde{K}_u) 
= \Tr(\hat{Q}_{kj}\,\tilde{K}_u)$ (since 
$C^{-1}K_{0,k}\,L_{0,j} = C^{-1/2}Y_k Y_j C^{-1/2} 
= \hat{Q}_{kj}$). Therefore
\[
\text{(I)}_{jk} = -\Tr(Y_j Y_k C) - \Tr(Y_j C Y_k) 
- \Tr(\hat{Q}_{jk}\,\tilde{K}_u) 
- \Tr(\hat{Q}_{kj}\,\tilde{K}_u).
\]
Substituting~\eqref{eq:key_trace_general}:
\begin{align}\label{eq:I_general}
\text{(I)}_{jk} &= -\Tr(Y_j Y_k C) - \Tr(Y_j C Y_k) 
+ \tfrac{1}{2}\Tr(Y_j Y_k C) + \tfrac{1}{2}\Tr(Y_k Y_j C)
\\
&= -\frac{1}{2}\big[\Tr(Y_j C Y_k) + \Tr(Y_k C Y_j) \big] \notag
\end{align}

\medskip\noindent
\textbf{Step 7: combining.}\;
From~\eqref{eq:jac_decomp}, \eqref{eq:II_general}, 
and~\eqref{eq:I_general}:
\[
\frac{\partial\psi_j}{\partial s_k}\bigg|_{(0,0)} 
= \text{(I)}_{jk} + \text{(II)}_{jk} 
= \frac{1}{2}\big[\Tr(Y_j C Y_k) + \Tr(Y_k C Y_j) \big].
\]

\medskip\noindent
\textbf{Positive definiteness.}\;
For any nonzero $\mathbf{v} = (v_1, \ldots, v_d) \in \bR^d$, 
set $Z = \sum_j v_j Y_j$ (nonzero, since the $Y_j$ are linearly 
independent). Then
\begin{align*}
\sum_{j,k} v_j v_k \,\frac{1}{2}\big[\Tr(Y_j C Y_k) + \Tr(Y_k C Y_j)\big]
&= \frac{1}{2}\big[\Tr(Z C Z) + \Tr(Z C Z)\big] 
\\
&= \Tr(Z^2 C) = \Tr(C^{1/2}\,Z^2\,C^{1/2}) > 0,
\end{align*}
since $Z \neq 0$ implies $Z^2 \neq 0$ (so $Z^2 \succeq 0$ is nonzero) 
and $C \succ 0$.
\end{proof}

\subsection{Non-singularity theorem} 


\begin{lemma}[Non-singularity at the trace minimizer]
\label{lem:nonsingularity}
Let $S$ be a matrix semicircle with self-adjoint covariance map 
$\eta\colon M_n(\bC) \to M_n(\bC)$ satisfying 
$\eta(C) = C^{-1}$ for some $C \succ 0$. Let 
$\Phi = S_{C^{1/2},C^{1/2}}(\eta)$, 
$\mathcal{V} := \ker(\Phi + \Id)|_{\Herm_n}$, 
$d := \dim\mathcal{V}$, and $\{Y_1, \ldots, Y_d\}$ be a basis 
of $\mathcal{V}$. Assume that
\begin{equation}\label{eq:FOC_lemma}
\Tr(Y_j C) = 0 \qquad \text{for all } j = 1, \ldots, d.
\end{equation}
Then:
\begin{enumerate}[label=\textup{(\roman*)}]
\item $S$ is non-singular with real-analytic density near 
  $x = 0$: the HFS solution $W(u)$ extends real-analytically 
  through $u = 0$ with $W(0^+) = C$.
\item The spectral density at $x = 0$ satisfies
  \[
  f(0) = \frac{1}{\pi}\,\tr(C) 
  \geq \frac{1}{\pi}\,\Capa(\eta)^{-1/(2n)} 
  = \frac{1}{\pi\sqrt{e}\,\Delta(S)}.
  \]
\end{enumerate}
\end{lemma}
\begin{proof}

Consider the equation
\[
\Psi(W, u) := W - (\eta(W) + uI)^{-1} = 0
\]
on $\Herm_n \times \bR$. Note that $\Psi(C, 0) = 0$ 
(since $\eta(C) = C^{-1}$), and the Fr\'echet derivative 
in $W$ at $(C, 0)$ is the linearization
$\LL = D_W\Psi|_{(C,0)}\colon K \mapsto K + C\,\eta(K)\,C$.
Let $d = \dim \ker(\Phi + \Id)|_{\Herm_n}$; by 
Proposition~\ref{propo:kernel_tangent_family}\,(\ref{item:tangent}), 
this equals $\dim \ker \LL$.

If $d = 0$, then $\LL = D_W\Psi(C,0)$ is invertible, and 
the implicit function theorem applies directly to 
$\Psi(W, u) = 0$ at $(C, 0)$.

If $d \geq 1$, by 
Proposition~\ref{propo:kernel_tangent_family}, the family 
$V_{\sv} = C^{1/2}\exp(\sum_j s_j Y_j)\,C^{1/2}$ solves 
$\Psi(V_{\sv}, 0) = 0$ for all $\sv$ near $0$. The 
Lyapunov--Schmidt reduction at $(C, 0)$ produces bifurcation 
equations $\phi_j(\sv, u) = u\,\psi_j(\sv, u)$. By 
Lemma~\ref{lem:transversality_general}, the 
hypothesis~\eqref{eq:FOC_lemma} gives $\psi_j(0, 0) = 0$ 
for all $j$ and the Jacobian 
$\partial\psi_j/\partial s_k|_{(0,0)}$ is positive definite. 
The implicit function theorem applied to the system 
$\psi_j(\sv, u) = 0$ at $(\sv, u) = (0, 0)$ yields a smooth 
curve $\sv(u)$ with $\sv(0) = 0$. Together with the auxiliary 
equation, this gives a real-analytic family of solutions 
$u \mapsto W(u)$ of $\Psi(W, u) = 0$ with $W(0) = C \succ 0$.

In both cases, for sufficiently small $u > 0$, $W(u)$ is 
strictly accretive (by continuity, since 
$W(0) = C \succ 0$). By the uniqueness theorem for the HFS 
equation~\cite[Theorem~2.1]{hfs2007}, this coincides with 
the HFS solution. Therefore $W(u) \to C$ as $u \to 0^+$, 
and $S$ is non-singular with 
$f(0) = \frac{1}{\pi}\tr(C)$.

For the lower bound: by 
Proposition~\ref{propo:kernel_tangent_family}, 
$\det(V_{\sv}) = \det(C)\,e^{\Tr(\sum s_j Y_j)} = \det(C)$ 
(since each $Y_j$ is traceless). By AM--GM, 
$\tr(C) \geq \det(C)^{1/n} = \Capa(\eta)^{-1/(2n)}$.
\end{proof}


\begin{theo}[Non-singularity for DS-scalable maps]
\label{theo:nonsingularity}
Let $S$ be a matrix semicircle with self-adjoint covariance map 
$\eta\colon M_n(\bC) \to M_n(\bC)$. Assume that the set
\[
\bS = \{W \succ 0 : \eta(W)\,W = I\}
\]
is nonempty. Then:
\begin{enumerate}[label=\textup{(\roman*)}]
\item $S$ is non-singular at $t_0 = 0$ with real-analytic 
  density near $x = 0$, and $f(0) = \frac{1}{\pi}\,\tr(C)$, 
  where $C \in \bS$ is the trace minimizer.
\item The trace minimizer is unique: there is a unique 
  $C \in \bS$ with $\tr(C) \leq \tr(W)$ for all $W \in \bS$.
\item The spectral density at $x = 0$ satisfies
  \[
  f(0) = \frac{1}{\pi}\,\tr(C) 
  \geq \frac{1}{\pi}\,\Capa(\eta)^{-1/(2n)} 
  = \frac{1}{\pi\sqrt{e}\,\Delta(S)}.
  \]
\end{enumerate}
\end{theo}

\begin{proof}
\textbf{(i).}\;
By Proposition~\ref{propo:trace_minimizer}, a trace minimizer 
$C \in \bS$ exists and satisfies $\Tr(Y_j C) = 0$ for all 
$Y_j \in \ker(\Phi + \Id)|_{\Herm_n}$. 
Lemma~\ref{lem:nonsingularity} then gives that $W(u)$ extends 
real-analytically through $u = 0$ with $W(0^+) = C$, so $S$ is 
non-singular with $f(0) = \frac{1}{\pi}\,\tr(C)$.

\medskip\noindent
\textbf{(ii).}\;
Let $C_1, C_2 \in \bS$ be any two trace minimizers. By 
part~(i) applied to each, the HFS solution satisfies both 
$W(0^+) = C_1$ and $W(0^+) = C_2$. Since $W(u)$ is unique 
for each $u > 0$, the limit is unique, so $C_1 = C_2$.

\medskip\noindent
\textbf{(iii).}\;
By Proposition~\ref{propo:kernel_tangent_family}, the family 
$V_{\sv} = C^{1/2}\exp(\sum_j s_j Y_j)\,C^{1/2}$ lies in 
$\bS$ with $\det(V_{\sv}) = \det(C)$ (since each $Y_j$ is 
traceless). By AM--GM,
\[
\tr(C) \geq \det(C)^{1/n} = \Capa(\eta)^{-1/(2n)},
\]
which gives the lower bound.
\end{proof}

\section{Proof of the main theorem} 
\label{sec:main}
\subsection{Non-singular $\Longrightarrow$ sym DS-scalable} 

\begin{lemma}
\label{lemma_W_symmetry}
Let $W(u)$ be the strictly accretive solution 
of~\eqref{eq:HMS} for $u \in D = \{\Re u > 0\}$. Then
\[
W(\bar{u}) = W(u)^*.
\]
In particular, $W(t) \succ 0$ for $t > 0$.
\end{lemma}

\begin{proof}
Since $\Re\, W \succ 0$, $W$ is invertible 
and~\eqref{eq:HMS} gives 
$\eta(W(u)) = W(u)^{-1} - uI$. (Note that this 
expresses $\eta(W)$ as a function of $W$ alone, so 
in particular $[W(u),\, \eta(W(u))] = 0$.)

Taking adjoints and using $\eta(W^*) = \eta(W)^*$:
\[
\big(W(u)^*\big)^{-1} 
= \bar{u}\,I + \eta\big(W(u)^*\big),
\]
which is exactly equation~\eqref{eq:HMS} at the 
point $\bar{u}$. By the uniqueness of the strictly 
accretive solution~\cite[Theorem~2.1]{hfs2007}, 
$W(\bar{u}) = W(u)^*$.
\end{proof}

\begin{propo}
\label{propo_regularity_implies_DS}
Let $S$ be a matrix semicircular variable that is non-singular at $t_0 = 0$. 
Then the covariance map $\eta_S$ is symmetrically DS-scalable.
\end{propo}

\begin{proof}
Let $W(u)$ denote the unique strictly accretive solution of
\begin{equation}\label{eq:HMS_again}
\eta\!\bigl(W(u)\bigr)\,W(u) + u\,W(u) = I,
\qquad \Re u > 0.
\end{equation}

\smallskip
\noindent\textbf{Step~1: Passing to the limit in the equation.}
By non-singularity at $t_0 = 0$, there exists a Stolz cone on which 
$W$ is bounded and has a non-tangential limit
\[
W_0 := \lim_{\substack{u \to 0 \\ \text{n.t.}}} W(u).
\]
Since $W$ is bounded on this cone, $uW(u) \to 0$ as $u \to 0$ 
non-tangentially. The linearity of $\eta$ gives 
$\eta(W(u)) \to \eta(W_0)$. Taking the non-tangential limit 
in~\eqref{eq:HMS_again}:
\begin{equation}
\label{equ_scaling_identity}
\eta(W_0)\,W_0 = I.
\end{equation}

\smallskip
\noindent\textbf{Step~2: $W_0 \succ 0$.}
The positive real axis $\{u = \tau : \tau > 0\}$ lies in every Stolz 
cone $\Gamma_\alpha(0)$, so the non-tangential limit agrees with the 
real-axis limit: $W_0 = \lim_{\tau \downarrow 0} W(\tau)$. 
By Lemma~\ref{lemma_W_symmetry}, $W(\tau) \succ 0$  for 
$\tau > 0$, and passing to the limit gives $W_0 \succeq 0$. 

Moreover, $W_0$ is invertible: if it were not, then 
$\operatorname{rank}(\eta(W_0)\,W_0) \leq \operatorname{rank}(W_0) < n$, 
contradicting~\eqref{equ_scaling_identity}. 
Therefore $W_0 \succ 0$.

\smallskip
\noindent\textbf{Step~3: Symmetric DS-scalability.}
Set $C := W_0 \succ 0$. Then \eqref{equ_scaling_identity} gives 
$\eta(C) = C^{-1}$, which is exactly symmetric DS-scalability 
of $\eta = \eta_S$ (Definition~\ref{defi_symmetric_DS}).
\end{proof}

\subsection{Proof of the main theorem} 
\label{section_proof_Thm_A}

\begin{theo}[= Theorem~A]
\label{theo_main_general}
Let $A = \sum_{i=1}^r A_i x_i$, $A_i^* = A_i$, be a Hermitian
matrix pencil, $S = \sum A_i \otimes s_i$ the associated matrix
semicircle, and $\eta\colon X \mapsto \sum A_i X A_i$ its
covariance map. The following are equivalent:
\begin{enumerate}[label=\textup{(\roman*)}]
\item $A$ is LR-semisimple.
\item $S$ is non-singular at $t = 0$.
\item $\eta$ is symmetrically DS-scalable.
\end{enumerate}
Moreover, when these hold, the spectral density satisfies
\[
  f(0) = \frac{1}{\pi}\,\tr(C)
  \;\geq\; \frac{1}{\pi}\,\Capa(\eta)^{-1/(2n)}
  \;=\; \frac{1}{\pi\sqrt{e}\,\Delta(S)},
\]
where $C$ is the unique trace minimizer of
$\bS = \{W \succ 0 : \eta(W)\,W = I\}$.
\end{theo}

\begin{proof}
Note that since the Kraus operators $A_i$ are Hermitian, 
$\eta$ is self-adjoint: $\eta^* = \eta$.

\medskip\noindent
\textbf{(i) $\Rightarrow$ (iii).}\;
If $A$ is LR-semisimple, then $\eta$ is symmetrically 
DS-scalable by Proposition~\ref{propo_LR_C}.

\medskip\noindent
\textbf{(iii) $\Rightarrow$ (i).}\;
Symmetric DS-scalability implies DS-scalability, and 
Proposition~\ref{propo_DS-scalable_LR-semisimple} gives 
that $A$ is LR-semisimple.

\medskip\noindent
\textbf{(iii) $\Rightarrow$ (ii).}\;
If $\eta$ is symmetrically DS-scalable, then 
$\bS \neq \emptyset$, and 
Theorem~\ref{theo:nonsingularity} gives that $S$ is 
non-singular at $t = 0$.

\medskip\noindent
\textbf{(ii) $\Rightarrow$ (iii).}\;
If $S$ is non-singular at $t = 0$, then $\eta$ is 
symmetrically DS-scalable by 
Proposition~\ref{propo_regularity_implies_DS}.

\medskip\noindent
The density formula and the lower bound are from 
Theorem~\ref{theo:nonsingularity}.
\end{proof}

\subsection{Corollaries} 
\label{section_corollaries}

\begin{coro}[= Theorem~B]
\label{theo_main_unsplittable}
  If $A$ is unsplittable with $A_i^\ast = A_i$, then $S$ is
  non-singular and $C$ is the \emph{unique} solution of
  $\eta(C) = C^{-1}$.
  \end{coro}
  
  \begin{proof}
  Since an unsplittable pencil is LR-semisimple, Theorem~\ref{theo_main_general} gives that $S$ is non-singular. An unsplittable pencil corresponds to an indecomposable CP map by Theorem \ref{theo_pencil_map_dictionary}(ii),  and uniqueness of scaling matrix $C$ follows from Lemma \ref{lem:indecomposable_sym_scalable}.
  \end{proof}
  
\begin{coro}[Direct sums of unsplittables]
  If $A$ is a direct sum of unsplittable Hermitian pencils,
  then $S$ is non-singular.
  \end{coro}

\begin{proof} Direct sum of unsplittables is LR-semisimple
  by definition. Main Theorem applies.
  \end{proof}
  
  \begin{coro}[Global regularity]
If $A$ is LR-semisimple, then the spectral density $f$ 
of $S$ is bounded and continuous on $\bR$, real-analytic 
in a neighborhood of $x = 0$, and satisfies $f(0) > 0$.
\end{coro}

\begin{proof}
The main theorem gives matrix-level non-singularity at 
$t_0 = 0$, hence real-analyticity and $f(0) > 0$. By 
Proposition~\ref{prop:singularity_localization}, $f$ is 
bounded at every $x \neq 0$. By 
Proposition~\ref{propo_algebraicity}, $f$ is real-analytic 
away from finitely many points. A bounded algebraic 
function at an isolated singularity extends continuously, 
so $f$ is continuous on all of $\bR$.
\end{proof}
  
  \begin{remark}[Doubly stochastic case]
When $\eta$ is already doubly stochastic, the Main Theorem 
applies with $C = I$ (the solution of 
$\eta(C) = C^{-1}$), giving $f(0) = \frac{1}{\pi}$. In 
fact, the spectral distribution of $S$ is the standard 
Wigner semicircle law: the ansatz $W(u) = w(u)\,I$ reduces 
Speicher's equation~\eqref{eq:HMS} to the scalar equation 
$w(u)^2 + u\,w(u) = 1$ (using unitality $\eta(I) = I$), 
whose unique solution with $\Re\, w > 0$ is 
$w(u) = \frac{-u + \sqrt{u^2 + 4}}{2}$. By HFS 
uniqueness, $W(u) = w(u)\,I$, and Stieltjes inversion 
recovers the semicircle density.
\end{remark}

\section{Discussion and open questions} 
\label{sec:discussion}

\begin{remark}\textbf{Geometric invariant theory perspective}
\label{rem:GIT}
The algebraic results of \S3 admit a natural interpretation in 
the framework of geometric invariant theory (GIT), as developed 
by \citet{ggow2020}; see also \cite{burgisser_etal2018, burgisser_etal2019} for a 
comprehensive treatment.

The group $G = GL_n(\bC) \times GL_n(\bC)$ acts on $r$-tuples 
of matrices $(A_1, \ldots, A_r) \in M_n(\bC)^r$ by
\[
(g, h) \cdot (A_1, \ldots, A_r) 
= (g\,A_1\,h^{-1},\, \ldots,\, g\,A_r\,h^{-1}).
\]
Under this action, LR-equivalence classes of pencils correspond 
to $G$-orbits. The key entries of the GIT dictionary are:
\begin{itemize}[leftmargin=2em]
\item \textbf{Capacity as a Kempf--Ness functional.}\; 
  The log-capacity $\log\Capa(\eta_A)$ plays the role of the 
  Kempf--Ness function, which in GIT measures the distance 
  of an orbit from being closed. Its critical points are the 
  DS-scalings of $\eta_A$.

\item \textbf{DS-scalability as polystability.}\;
  A pencil $A$ is DS-scalable (equivalently, 
  $\Capa(\eta_A) > 0$ and the capacity infimum is attained) 
  if and only if the $G$-orbit of $(A_1, \ldots, A_r)$ is 
  closed in the semistable locus --- the GIT notion of 
  polystability. In this language, 
  Proposition~\ref{propo_DS-scalable_LR-semisimple} says that 
  LR-semisimplicity is equivalent to polystability.

\item \textbf{Operator Sinkhorn iteration as gradient descent.}\;
  The alternating row/column normalization of 
  Definition~\ref{defi_OSI} can be interpreted as Riemannian 
  gradient descent on the Kempf--Ness function over the 
  symmetric space $GL_n/U_n$. The GGOW algorithm 
  \cite{ggow2020} provides polynomial-time convergence 
  guarantees for this iteration.
\end{itemize}

We have not used this perspective in the proofs, but it 
provides useful geometric intuition for the algebraic 
results of~\S3.
\end{remark}

\medskip
The main theorem characterizes non-singularity at $x = 0$ 
for unbiased self-adjoint matrix semicircles. Several 
natural extensions remain open.

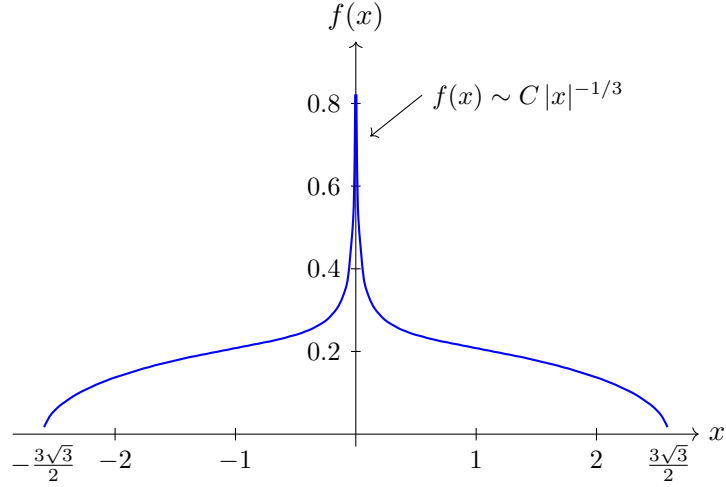
\begin{figure}[ht]
\centering
\begin{tikzpicture}[xscale=1.6, yscale=5.5]
  \draw[->] (-2.85,0) -- (2.85,0) node[right] {$x$};
  \draw[->] (0,-0.03) -- (0,0.95) node[above] {$f(x)$};
  \foreach \x in {-2,-1,1,2}
    \draw (\x,0.015) -- (\x,-0.015) node[below]{\small $\x$};
  \foreach \y in {0.2, 0.4, 0.6, 0.8}
    \draw (0.04,\y) -- (-0.04,\y) node[left]{\small $\y$};
  \draw[thick, blue] plot[smooth] coordinates {
    (-2.590,0.017) (-2.546,0.042) (-2.503,0.057)
    (-2.459,0.069) (-2.416,0.078) (-2.372,0.087)
    (-2.329,0.095) (-2.285,0.102) (-2.242,0.108)
    (-2.198,0.114) (-2.155,0.120) (-2.111,0.125)
    (-2.068,0.130) (-2.024,0.135) (-1.981,0.139)
    (-1.937,0.143) (-1.894,0.147) (-1.850,0.151)
    (-1.806,0.155) (-1.763,0.159) (-1.719,0.162)
    (-1.676,0.166) (-1.632,0.169) (-1.589,0.172)
    (-1.545,0.175) (-1.502,0.178) (-1.458,0.181)
    (-1.415,0.184) (-1.371,0.187) (-1.328,0.189)
    (-1.284,0.192) (-1.241,0.194) (-1.197,0.197)
    (-1.154,0.199) (-1.110,0.202) (-1.066,0.204)
    (-1.023,0.207) (-0.979,0.209) (-0.936,0.212)
    (-0.892,0.214) (-0.849,0.217) (-0.805,0.219)
    (-0.762,0.222) (-0.718,0.224) (-0.675,0.227)
    (-0.631,0.230) (-0.588,0.233) (-0.544,0.237)
    (-0.501,0.240) (-0.457,0.244) (-0.414,0.249)
    (-0.370,0.254) (-0.326,0.261) (-0.283,0.268)
    (-0.239,0.277) (-0.196,0.290) (-0.152,0.306)
    (-0.109,0.332) (-0.065,0.379)
    (-0.022,0.519) (-0.010,0.660) (-0.005,0.822)
  };
  \draw[thick, blue] plot[smooth] coordinates {
    (0.005,0.822) (0.010,0.660) (0.022,0.519)
    (0.065,0.379) (0.109,0.332) (0.152,0.306)
    (0.196,0.290) (0.239,0.277) (0.283,0.268)
    (0.326,0.261) (0.370,0.254) (0.414,0.249)
    (0.457,0.244) (0.501,0.240) (0.544,0.237)
    (0.588,0.233) (0.631,0.230) (0.675,0.227)
    (0.718,0.224) (0.762,0.222) (0.805,0.219)
    (0.849,0.217) (0.892,0.214) (0.936,0.212)
    (0.979,0.209) (1.023,0.207) (1.066,0.204)
    (1.110,0.202) (1.154,0.199) (1.197,0.197)
    (1.241,0.194) (1.284,0.192) (1.328,0.189)
    (1.371,0.187) (1.415,0.184) (1.458,0.181)
    (1.502,0.178) (1.545,0.175) (1.589,0.172)
    (1.632,0.169) (1.676,0.166) (1.719,0.162)
    (1.763,0.159) (1.806,0.155) (1.850,0.151)
    (1.894,0.147) (1.937,0.143) (1.981,0.139)
    (2.024,0.135) (2.068,0.130) (2.111,0.125)
    (2.155,0.120) (2.198,0.114) (2.242,0.108)
    (2.285,0.102) (2.329,0.095) (2.372,0.087)
    (2.416,0.078) (2.459,0.069) (2.503,0.057)
    (2.546,0.042) (2.590,0.017)
  };
  \draw[<-, thin] (0.12,0.72) -- (0.55,0.82)
    node[right]{\small $f(x) \sim C\,|x|^{-1/3}$};
  \node[below] at (-2.598,0) {\small $-\frac{3\sqrt{3}}{2}$};
  \node[below] at (2.598,0) {\small $\frac{3\sqrt{3}}{2}$};
\end{tikzpicture}
\caption{Spectral density of the matrix semicircle 
$S = A_1 \otimes s_1 + A_2 \otimes s_2$ for the 
non-LR-semisimple pencil of 
Example~\textup{\ref{exa:full_not_LR}}. The density 
has an integrable $|x|^{-1/3}$ cusp at the origin.}
\label{fig:singular_density}
\end{figure}

\medskip\noindent
\textbf{Singularities when $A$ is not LR-semisimple.}\;
When the pencil is full but not LR-semisimple, the main 
theorem implies that $S$ is singular at $t = 0$. The 
following example shows that the singularity can take the 
form of an integrable algebraic cusp.

\begin{exa}[Singularity for a non-LR-semisimple pencil]
\label{exa:singular_density}
Consider the pencil from Example~\ref{exa:full_not_LR}:
\[
A_1=\begin{pmatrix}0&1\\1&0\end{pmatrix},\qquad
A_2=\begin{pmatrix}0&0\\0&1\end{pmatrix}.
\]
This pencil is full but not LR-semisimple.

Let
\[
G(z)=\begin{pmatrix}a(z)&w(z)\\ w(z)&c(z)\end{pmatrix},
\qquad z\in\C^+,
\]
be the matrix Cauchy transform, so that
\[
zG(z)=I+\eta(G(z))G(z),
\qquad \Im G(z)\prec 0.
\]
That is, $G$ is the Herglotz solution of Speicher's equation. A direct computation gives
\[
\eta\!\left(\begin{pmatrix}a&w\\ w&c\end{pmatrix}\right)
=
\begin{pmatrix}c&w\\ w&a+c\end{pmatrix}.
\]
Let $U=\diag(1,-1)$. Since $UA_1U=-A_1$ and $UA_2U=A_2$, we have
\[
\eta(UXU)=U\eta(X)U \qquad (X\in M_2(\C)).
\]
Hence $UG(z)U$ satisfies the same equation as $G(z)$, and by uniqueness of the Herglotz solution,
$UG(z)U=G(z)$. Therefore $w(z)=0$.

Thus $G(z)=\diag(a(z),c(z))$, and Speicher's equation reduces to
\[
a(z-c)=1,\qquad c(z-a-c)=1.
\]
Set $v:=z-c$. Then $a=1/v$ and $c=z-v$, so
\[
v^3-zv^2+z=0.
\]
Its discriminant is
\[
\Delta(z)=z^2(4z^2-27),
\]
hence the spectral density is supported on
\[
\left[-\frac{3\sqrt3}{2},\,\frac{3\sqrt3}{2}\right].
\]

For $|x|<\frac{3\sqrt3}{2}$, let $v(x)=\alpha(x)+i\beta(x)$ be the boundary-value root selected by
$\Im G(x+i0)\prec 0$; then $\beta(x)>0$. Since
\[
\tr G(x+i0)=\frac12\left(\frac1{v(x)}+x-v(x)\right),
\]
Stieltjes inversion gives
\[
f(x)
=
-\frac1\pi \Im \tr G(x+i0)
=
\frac{1}{2\pi}\Im\!\left(v(x)-\frac1{v(x)}\right)
=
\frac{\beta(x)}{2\pi}\left(1+\frac1{|v(x)|^2}\right).
\]
In particular, $f$ is real-analytic on
\[
\left(-\frac{3\sqrt3}{2},0\right)\cup
\left(0,\frac{3\sqrt3}{2}\right).
\]

Near $x=0$, the cubic gives $v(x)^3\sim -x$, hence
$|v(x)|\asymp |x|^{1/3}$ and $\beta(x)\asymp |x|^{1/3}$. Therefore
\[
f(x)\asymp |x|^{-1/3}\qquad (x\to 0).
\]
Thus $f$ has an integrable cusp singularity at $0$, in contrast with the bounded density from Theorem~\ref{theo_main_general} for LR-semisimple pencils.
\end{exa}

Is this representative of the general case? By the 
algebraicity of the Cauchy transform 
(Proposition~\ref{propo_algebraicity}), the singularity 
at $0$ is always algebraic, so $f(u) \sim C\,|u|^{-\alpha}$ 
for some rational $\alpha \in (0, 1)$ (integrability forces 
$\alpha < 1$ since there is no atom at $0$ for full pencils). 
It would be interesting to determine which exponents 
$\alpha$ can occur and whether they are controlled by 
the pencil structure.  

\medskip\noindent
\textbf{Biased matrix semicircles.}\;
For $S = A_0 + \sum_{i=1}^r A_i \otimes s_i$ with 
$A_0 \neq 0$, the special role of $x = 0$ is lost. 
The natural question is whether singularities of the 
density on the entire real line can be characterized 
algebraically in terms of the pencil 
$A_0 + \sum A_i x_i$.

\medskip\noindent
\textbf{Non-self-adjoint case.}\;
When the matrices $A_i$ are not required to be Hermitian, 
the spectral distribution is replaced by the Brown measure 
of $S$, which is a probability measure on $\bC$. The 
self-adjointness $\eta^* = \eta$ was used in several places 
(the geodesic reflection theorem, the symmetry 
$W(\bar{u}) = W(u)^*$, and the Lyapunov--Schmidt analysis), 
so new ideas would be needed.


\newpage
\appendix


\section{Tools from operator theory} 

\begin{lemma}[Kadison--Schwarz inequality]
\label{lem:kadison_schwarz}
Let $\Phi: M_n(\bC) \to M_n(\bC)$ be a unital completely positive map. 
Then for any $A \in M_n(\bC)$,
\[
\Phi(A)^*\Phi(A) \leq \Phi(A^*A).
\]
In particular, if $L$ is Hermitian, then $\Phi(L)^2 \leq \Phi(L^2)$.
\end{lemma}

\begin{proof}
By the Choi--Kraus theorem, write $\Phi(X) = \sum_{i=1}^r K_i X K_i^*$.
Define the isometry $V: \bC^n \to (\bC^n)^{\oplus r}$ by 
$Vx = (K_1^* x, \ldots, K_r^* x)$. Unitality gives $V^*V = I_n$, 
and $\Phi(X) = V^*(X \otimes I_r)V$. Let $P = VV^*$ denote the 
orthogonal projection onto the range of $V$. Then
\begin{align*}
\Phi(A^*A) - \Phi(A)^*\Phi(A)
&= V^*(A^*A \otimes I)V - V^*(A^* \otimes I)P(A \otimes I)V \\
&= V^*\bigl[(A^* \otimes I)(I - P)(A \otimes I)\bigr]V \\
&= V^*\bigl[(I - P)(A \otimes I)\bigr]^*\bigl[(I - P)(A \otimes I)\bigr]V 
\;\geq\; 0,
\end{align*}
where we used $(I - P) = (I - P)^2 = (I-P)^*$.
\end{proof}

\begin{theo}[Choi's multiplicative domain theorem]
\label{theo_choi}
Let $\Phi: M_n(\bC) \to M_n(\bC)$ be a unital completely positive map 
and let $H \in M_n(\bC)$ be Hermitian. If equality holds in the 
Kadison--Schwarz inequality, 
\[
\Phi(H^2) = \Phi(H)^2,
\]
then $\Phi(XH) = \Phi(X)\Phi(H)$ and $\Phi(HX) = \Phi(H)\Phi(X)$ 
for all $X \in M_n(\bC)$. In particular, $\Phi$ restricts to a 
$*$-homomorphism on the $C^*$-algebra $C^*(H)$ generated by $H$ 
and $I$.
\end{theo}

\noindent
\begin{remark}For general (not necessarily Hermitian) elements,
\cite{choi74} establishes the one-sided conclusion 
$\Phi(XA) = \Phi(X)\Phi(A)$ from the hypothesis 
$\Phi(A^*A) = \Phi(A^*)\Phi(A)$, requiring only $2$-positivity 
of~$\Phi$.
\end{remark}

\begin{proof}
Let $V$ and $P = VV^*$ be as in the proof of 
Lemma~\ref{lem:kadison_schwarz}. The equality 
$\Phi(H^2) = \Phi(H)^2$ means
\[
V^*\bigl[(I - P)(H \otimes I)\bigr]^*\bigl[(I - P)(H \otimes I)\bigr]V = 0.
\]
Since $V$ is an isometry, this forces 
$(I - P)(H \otimes I)V = 0$, i.e.,
\begin{equation}
\label{equ_intertwining}
(H \otimes I)\,V = P\,(H \otimes I)\,V = V\,\Phi(H).
\end{equation}
Since $H = H^*$, taking adjoints gives 
$V^*(H \otimes I) = \Phi(H)\,V^*$. Therefore, for any 
$X \in M_n(\bC)$:
\begin{align*}
\Phi(XH) &= V^*(X \otimes I)(H \otimes I)\,V 
= V^*(X \otimes I)\,V\,\Phi(H) = \Phi(X)\,\Phi(H), \\
\Phi(HX) &= V^*(H \otimes I)(X \otimes I)\,V 
= \Phi(H)\,V^*(X \otimes I)\,V = \Phi(H)\,\Phi(X). 
\end{align*}
\end{proof}

\section{Log-determinant and capacity tools} 


\begin{lemma}[Derivatives of log-determinant along a curve]
\label{lemma_logdet_derivatives}
Let $\phi: \mathbb{R} \to \mathcal{P}_n$ be a smooth curve of positive definite matrices. Then
\begin{align}
\frac{d}{dt}\log\det(\phi(t)) &= \Tr[\phi(t)^{-1}\,\phi'(t)], \label{equ_logdet_first} \\[4pt]
\frac{d^2}{dt^2}\log\det(\phi(t)) &= \Tr[\phi(t)^{-1}\,\phi''(t)] - \Tr\bigl[(\phi(t)^{-1}\,\phi'(t))^2\bigr]. \label{equ_logdet_second}
\end{align}
\end{lemma}

\begin{proof}
For the first derivative, the expansion 
$\det(I + \epsilon A) = 1 + \epsilon\,\Tr(A) + O(\epsilon^2)$ gives
\[
\frac{d}{dt}\log\det(\phi) 
= \frac{1}{\det(\phi)}\frac{d}{dt}\det(\phi) 
= \Tr[\phi^{-1}\,\phi'].
\]
For the second derivative, differentiating $\phi\,\phi^{-1} = I$ yields
$\frac{d}{dt}(\phi^{-1}) = -\phi^{-1}\,\phi'\,\phi^{-1}$.
Applying the product rule to~\eqref{equ_logdet_first}:
\begin{align*}
\frac{d^2}{dt^2}\log\det(\phi)
&= \Tr\!\left[\frac{d}{dt}(\phi^{-1})\cdot\phi'\right] + \Tr[\phi^{-1}\,\phi''] \\
&= -\Tr[\phi^{-1}\,\phi'\,\phi^{-1}\,\phi'] + \Tr[\phi^{-1}\,\phi''] \\
&= \Tr[\phi^{-1}\,\phi''] - \Tr\bigl[(\phi^{-1}\,\phi')^2\bigr]. \qedhere
\end{align*}
\end{proof}

\begin{lemma}[Convexity of log-determinant along CP image of a geodesic]
\label{lemma_logdet_convexity}
Let $\eta: M_n(\bC) \to M_n(\bC)$ be a completely positive map, 
let $B \succ 0$, and suppose $\eta(B^t) > 0$ for all $t$ in an interval. 
Set $L = \log B$ and define $h(t) = \log\det(\eta(B^t))$. Then
\begin{equation}
\label{equ_logdet_convexity}
h''(t) = \Tr[\Psi_t(L^2)] - \Tr[\Psi_t(L)^2] \geq 0,
\end{equation}
where $\Psi_t: M_n(\bC) \to M_n(\bC)$ is the unital CP map
\[
\Psi_t(X) = \eta(B^t)^{-1/2}\,\eta(B^{t/2}\, X\, B^{t/2})\,\eta(B^t)^{-1/2}.
\]
\end{lemma}

\begin{proof}
Set $\phi(t) = \eta(B^t)$. Since $L$ commutes with all powers of $B$,
\[
\phi'(t) = \eta(B^t L), \qquad \phi''(t) = \eta(B^t L^2).
\]
By Lemma~\ref{lemma_logdet_derivatives},
\begin{equation}
\label{equ_h_second_raw}
h''(t) = \Tr[\phi^{-1}\,\phi''] - \Tr[(\phi^{-1}\,\phi')^2].
\end{equation}
We now rewrite each term using $\Psi_t$. Since $L$ commutes 
with $B^{t/2}$, we have $B^{t/2} L^k B^{t/2} = B^t L^k$ for $k = 1, 2$, 
and therefore
\[
\Psi_t(L^k) = \phi^{-1/2}\,\eta(B^t L^k)\,\phi^{-1/2}.
\]
For the first term in~\eqref{equ_h_second_raw}:
\[
\Tr[\Psi_t(L^2)] 
= \Tr[\phi^{-1/2}\,\eta(B^t L^2)\,\phi^{-1/2}] 
= \Tr[\phi^{-1}\,\phi''].
\]
For the second term:
\[
\Psi_t(L)^2 
= \phi^{-1/2}\,\phi'\,\phi^{-1}\,\phi'\,\phi^{-1/2},
\]
so $\Tr[\Psi_t(L)^2] = \Tr[(\phi^{-1}\,\phi')^2]$. 
This establishes~\eqref{equ_logdet_convexity}.

The map $\Psi_t$ is CP (as a composition of $\eta$ with conjugation maps) 
and unital: $\Psi_t(I) = \phi^{-1/2}\,\eta(B^t)\,\phi^{-1/2} = I$.
The Kadison--Schwarz inequality (Lemma~\ref{lem:kadison_schwarz}) gives 
$\Psi_t(L)^2 \preceq \Psi_t(L^2)$, hence $h''(t) \geq 0$.
\end{proof}

\begin{lemma}[Gradient of log-determinant composed with a linear map]
\label{lemma_gradient_logdet}
Let $\eta: M_n(\bC) \to M_n(\bC)$ be a linear map with adjoint $\eta^*$
(with respect to the inner product $\langle X, Y \rangle = \Tr(XY^*)$).
For any $X \succ 0$ with $\eta(X) \succ 0$,
\[
\nabla_X \log\det(\eta(X)) = \eta^*\bigl(\eta(X)^{-1}\bigr).
\]
\end{lemma}

\begin{proof}
Recall that for nonsingular $Y > 0$,
\[
\frac{d}{dt}\bigg|_{t=0} \log\det(Y + tH) = \Tr(Y^{-1} H),
\]
so $\nabla_Y \log\det(Y) = Y^{-1}$ under the trace inner product.

By the chain rule, since $\eta$ is linear,
\begin{align*}
\frac{d}{dt}\bigg|_{t=0} \log\det(\eta(X + tH))
&= \Tr\bigl(\eta(X)^{-1}\,\eta(H)\bigr)
= \bigl\langle \eta(X)^{-1},\, \eta(H) \bigr\rangle
\\
&= \bigl\langle \eta^*(\eta(X)^{-1}),\, H \bigr\rangle,
\end{align*}
where the last step is the definition of $\eta^*$.
\end{proof}

\section{Full $\Leftrightarrow$ rank non-decreasing (proof)} 
\label{proof_full_eq_RND}
Here we give a simple proof of Theorem~\ref{theo_pencil_map_dictionary}(i).
This equivalence corresponds to the implication (5)\,$\Leftrightarrow$\,(7) 
in Theorem~1.4 of \cite{ggow2020}, which is stated there without proof.  
We include a short argument for completeness.

\begin{proof}
Let $A = \sum_{i=1}^r A_i x_i$ be a matrix pencil of size~$n$ and 
let $\eta_A(X) = \sum_{i=1}^r A_i\, X\, A_i^\ast$ be its CP companion map.
The key observation is that for any positive semidefinite $B \succeq 0$ with 
$\mathcal{U} := \range B$, one has
\begin{equation}\label{eq:image_identity}
  \range (A_i\, B\, A_i^\ast) \;=\; A_i(\mathcal{U})
  \qquad\text{for each } i.
\end{equation}
Indeed, writing $B = B^{1/2} B^{1/2}$ gives 
$A_i B A_i^\ast = (A_i B^{1/2})(A_i B^{1/2})^\ast$, and since 
$\range(CC^\ast) = \range C$ for any matrix~$C$, 
we obtain 
$\range (A_i B A_i^\ast) = \range (A_i B^{1/2}) 
= A_i(\range B^{1/2} ) = A_i(\mathcal{U})$.
We also recall that for psd matrices $0 \leq P \leq Q$, one has 
$\range P  \subseteq \range Q$.

\medskip
\noindent
$(\Longrightarrow)$\; 
\textbf{Full implies rank non-decreasing.}\;
We prove the contrapositive.  Suppose $\eta_A$ is rank-decreasing: 
there exists $B \succeq 0$ with $\rank \eta_A(B) < \rank B$.  
Set $\mathcal{U} = \range B$ and $\mathcal{W} = \range \eta_A(B)$.
Since each summand $A_i B A_i^\ast$ is positive semidefinite and 
$0 \leq A_i B A_i^\ast \leq \eta_A(B)$, we have 
\[
\range (A_i B A_i^\ast) \subseteq \range \eta_A(B) = \mathcal{W}.
\] 
By~\eqref{eq:image_identity}, this gives $A_i(\mathcal{U}) \subseteq \mathcal{W}$ 
for every~$i$.  Moreover,
\[
  \dim \mathcal{U} = \rank B 
  > \rank \eta_A(B) = \dim \mathcal{W}.
\]
Hence $(\mathcal{U}, \mathcal{W})$ is a shrunk subspace, so $A$ is not full.

\medskip
\noindent
$(\Longleftarrow)$\; 
\textbf{Not full implies rank-decreasing.}\;
Suppose $A$ is not full: there exist subspaces 
$\mathcal{U}, \mathcal{W} \subset \mathbb{C}^n$ with 
$A_i(\mathcal{U}) \subseteq \mathcal{W}$ for all~$i$ and 
$\dim \mathcal{U} > \dim \mathcal{W}$.  
Let $B = P_{\mathcal{U}}$ be the orthogonal projection onto~$\mathcal{U}$.  
Then $B \succeq 0$, $\rank B = \dim \mathcal{U}$, and 
$\range B = \mathcal{U}$.  
By~\eqref{eq:image_identity}, 
$\range (A_i B A_i^\ast) = A_i(\mathcal{U}) \subseteq \mathcal{W}$ 
for each~$i$, so
\[
  \range \eta_A(B) 
  = \range \Bigl(\sum_{i=1}^r A_i B A_i^\ast\Bigr) 
  \subseteq \mathcal{W},
\]
where the inclusion holds because $\mathcal{W}$ is a subspace and 
$\range (A_i B A_i^\ast) \subseteq \mathcal{W}$ for each~$i$.
Therefore,
\[
  \rank \eta_A(B)
  \leq \dim \mathcal{W} 
  < \dim \mathcal{U} 
  = \rank B,
\]
and $\eta_A$ is rank-decreasing.
\end{proof}

\begin{remark}
The equivalence above uses only the definition of fullness 
via shrunk subspaces and elementary properties of positive semidefinite 
matrices.  The deeper equivalence between fullness and full inner rank 
(invertibility over the free skew field), which is due to 
\cite{cohn85, cohn95}, is not needed here.
\end{remark}



\end{document}